\documentclass{amsart}

\usepackage{amsthm}
\usepackage{amssymb}
\usepackage{amsmath}
\usepackage{mathtools}
\usepackage{pdfpages}
\usepackage{float}
\usepackage[colorlinks=true,pdftex,unicode=true,linktocpage,bookmarksopen,hypertexnames=false]{hyperref}
\usepackage{tikz-cd}
\usepackage{caption}

\usepackage[sorting=nyt,style=alphabetic, url=false, doi=false,  maxbibnames=4]{biblatex}
\bibliography{refs}

\DeclareMathOperator{\Aut}{Aut}

\DeclareMathOperator{\id}{id}

\DeclareMathOperator{\Ret}{Ret}

\DeclareMathOperator{\Sym}{Sym}

\DeclareMathOperator{\Hol}{Hol}
\DeclareMathOperator{\Stab}{Stab}

\makeatletter
\numberwithin{equation}{section}
\numberwithin{figure}{section}
\numberwithin{table}{section}
\newtheorem{thm}{Theorem}[section]
\newtheorem*{thm*}{Theorem}
\newtheorem{lem}[thm]{Lemma}
\newtheorem{cor}[thm]{Corollary}
\newtheorem{pro}[thm]{Proposition}

\theoremstyle{definition} 
\newtheorem{defn}[thm]{Definition}

\newtheorem{exa}[thm]{Example}

\makeatother

\title[Solutions to the Pentagon Equation]{Bijective solutions to the Pentagon Equation}
\author{I. Colazzo}
\address{School of Mathematics, University of Leeds, Leeds, UK. ORCID: 0000-0002-2713-0409}
\email{I.Colazzo@leeds.ac.uk}
\author{J. Okni\'nski}
\address{Institute of Mathematics, University of Warsaw, Banacha 2, 02-097 Warsaw, Poland. ORCID: 0000-0002-2434-3425}
\email{okninski@mimuw.edu.pl}
\author{A. Van Antwerpen}
\address{Department of Mathematics and Statistics, National University of Ireland - Maynooth, Maynooth, Ireland. ORCID: 0000-0001-7619-6298}
\email{arne.vanantwerpen@mu.ie}

\subjclass[2020]{Primary: 81R12, 20M99, Secondary: 16T25, 20E22, 81R50 }
\keywords{Pentagon equation, Set-theoretic solution, Matched product of groups, Semigroup}

\begin{document}

\begin{abstract}
A complete classification of all finite bijective set-theoretic solutions $(S,s)$ to the Pentagon Equation is obtained. First, it is shown that every such solution determines a semigroup structure on the set $S$ that is the direct product $E\times G$ of a semigroup of left zeros $E$ and a group $G$. Next, we prove that this leads to a decomposition of the set $S$ as a Cartesian
product $X\times A\times G$, for some sets $X,A$ and to the discovery of a hidden group structure on $A$. Then an unexpected structure of a matched product of groups $A,G$ is found such that the solution $(S,s)$ can be explicitly described as a lift of a solution determined on the set $A\times G$ by this matched product of groups. Conversely, every matched product of groups leads to a family of solutions arising in this way. Moreover, a simple criterion for the isomorphism of two solutions is obtained. These results significantly extend those of Colazzo, Jespers, and Kubat, in their treatment of the involutive case.  Furthermore, connections to solutions to the Yang--Baxter Equation and to the theory of skew braces are uncovered. The latter motivate a further investigation of the relationship between solutions to the Yang--Baxter Equation and the Pentagon Equation at the set-theoretical level.
\end{abstract}
 
\maketitle

\section{Introduction}

A linear map $v : V \otimes V \to V \otimes V $ on the tensor square of a vector space $V$ is
a solution to the \emph{Pentagon Equation} if, on $V\otimes  V \otimes V$, one has $$v_{23}v_{13}v_{12} = v_{12}v_{23},$$ where each $v_{ij}$ is the map
applying $v$ to the $(i,j)$-component and acting as the identity on the remaining one. If $S$ is a nonempty set and $s\colon S\times S\to
S\times S$ is a map, then the pair $(S,s)$ is said to be a set-theoretic solution to the
Pentagon Equation if the equality
\[
s_{23}s_{13}s_{12}=s_{12}s_{23}
\]
holds on $S\times S\times S$, with the same notation agreement as above. Namely, if $\tau: S\times S\to S\times S$, $\tau(x,y)=(y,x)$ is the flip map then $s_{12}=s\times \id$,
$s_{13}=(\tau \times \id)(\id \times s)(\tau \times \id)$ and $s_{23}=\id \times s$. Clearly, every set-theoretic solution $(S,s)$ to the Pentagon Equation extends linearly to a unique solution $(V,v)$ on the vector space $V$ with basis $S$.

The Pentagon Equation has appeared in several areas of mathematics, especially in the theory of quantum groups and Hopf algebras, as well as in various contexts of mathematical physics. In particular, these include: operators on Hilbert spaces with applications to  duality results on compact groups \cite{BaSk93,BaSk98}, three dimensional integrable systems and the so called tetrahedron equation \cite{Ma94}, conformal field theory  \cite{MR1002038}, the Heisenberg double and classification problems for Hopf algebras \cite{Mi04}, Cuntz algebras  \cite{MR2679972}.
The Pentagon Equation also appears in low-dimensional topology in the construction of link invariants in the work of Kashaev \cite{zbMATH01628545}.
In \cite{MR2349626}, the set-theoretic version of the Pentagon Equation was applied to the so called symmetrically factorizable Lie groups, while in \cite{Ka2011} it  was considered in the context of the discrete Liouville equation. 
Moreover, nontrivial relations between the Pentagon Equation and another important equation of mathematical physics, the Yang--Baxter Equation, were indicated and investigated by several authors. In particular, in \cite{Ka96} it is shown that the Pentagon Equation plays the same role for the Heisenberg double as the Quantum Yang–Baxter Equation does for the Drinfeld double. The context of Hopf algebras and of the relations between the Yang--Baxter and Pentagon Equations were studied also in \cite{MR1273649}. 
Solutions to the Pentagon Equation also have appeared with different names. For example, a unitary operator on the tensor square of a Hilbert space is said to be multiplicative if it is a solution to the
Pentagon Equation \cite{BaSk93, MR1369908}; while in \cite{St98}, for a fixed braided monoidal category $A$, a map $V : A \otimes A \to A \otimes A$ is said to be a fusion operator if it satisfies the Pentagon
Equation.
Recently, set-theoretic solutions to the Pentagon Equation received a lot of attention. They have been approached via new methods, including methods of semigroup theory, with an effort to describe and classify certain special classes of such solutions, see   \cite{MR4152351,MR4170296,castelli2024commutative, Jiang2005}. 
In particular, Colazzo, Jespers and Kubat \cite{MR4170296} completely classified involutive set-theoretic solutions $(S,s)$ (meaning that $s^2 =\id$) to the Pentagon Equation, while Castelli \cite{castelli2024commutative} more recently obtained a description of some other special classes of bijective set-theoretic solutions.

In this paper we deal with arbitrary finite bijective set-theoretic solutions $(S,s)$ to the Pentagon Equation and we provide a complete classification of all such solutions. In particular, our results provide far-reaching extensions of those obtained in \cite{MR4170296} and  \cite{castelli2024commutative}. 

Surprisingly, matched products of groups play a crucial role in our main results. These are also known as Zappa--Sz\'{e}p products or bicrossed products of groups, and they arise as exact factorizations of groups, $G=AB$ with $A,B$ subgroups of $G$ and $A\cap B=\{1\}$. For a survey, with an emphasis on their role in the theory of the Yang--Baxter Equation, we refer to the paper of Takeuchi \cite{MR2024436}.

Our main theorem can be formulated as follows (see Section~\ref{sect_matched} for the definition of a matched pair of groups).

\begin{thm} \label{main}
     Let $(S,s)$ be a finite bijective set-theoretic solution to the Pentagon Equation. Then there
exists a matched pair of groups $(A,G,\sigma, \delta)$, a set $X$
and maps $\varphi_a\in \Sym(X)$, for all $a\in A$, such that $S$,
as a set, is the Cartesian product  $X\times A \times G$ and $s$
is of the following form
     \begin{align}\label{eq:star}\tag{$\star$}
         s((\alpha,a,x),(\beta,b,y))
         =((\alpha,a,xy),(\varphi_{b(\delta_x(a))^{-1}}\varphi_b(\beta),
b(\delta_x(a))^{-1},\sigma_{\delta_x(a)}(y)))
     \end{align}
     for $\alpha ,\beta \in X, a,b \in A$ and $x,y\in G$.
     Moreover, for any matched pair of groups $(A,G,\sigma, \delta)$, a
set $X$ and maps $\varphi_a\in \Sym(X)$, defined for all $a\in A$,
the map $s$ defined on $S= X\times A \times G$ in \eqref{eq:star}
is a bijective set-theoretic solution to the Pentagon Equation.
\end{thm}

There is also a natural notion of a homomorphism of set-theoretic solutions to
the Pentagon Equation (see Definition~\ref{def:iso}) and  the proof of Theorem~\ref{main} yields a transparent criterion for when two set-theoretic solutions are isomorphic. In fact, up to isomorphism the family of permutations $\{\varphi_a\}$ is not essential, i.e. every finite bijective solution is isomorphic to one of the form
\eqref{eq:star} with $\varphi_a=\id_X$ for all $a\in A$. In particular, up to isomorphism we may assume $\varphi_a=\id_X$ for all $a$, and then
\eqref{eq:star} becomes
\begin{align*}
         s((\alpha,a,x),(\beta,b,y))
         =((\alpha,a,xy),(\beta,
b(\delta_x(a))^{-1},\sigma_{\delta_x(a)}(y))).
\end{align*}
In addition, isomorphism classes of finite bijective solutions to the Pentagon Equation
are in a bijection with pairs consisting of a positive integer and an isomorphism class of
matched pairs of groups (see Corollary~\ref{cor:classif-iso}).
It is worth stressing that Theorem~\ref{main} seems quite surprising, since the analogous problem on (describing all) bijective set-theoretic solutions to the Yang--Baxter Equation has proved to be extremely difficult and remains wide open, see \cite{zbMATH07919707} and references therein.

% \vspace{-0.6em}

By convention, if $(S,s)$ is a set-theoretic solution to the Pentagon Equation, we write
$s(x,y)=(xy, \theta_x(y))$, for a map $S\times S\to S$,
$(x,y)\mapsto x\cdot y$ and certain maps $\theta_x\colon S\to S$,
for $x\in S$. It is well known that $(S,\cdot )$ is a semigroup,
\cite{MR1637789}.

The proof of our main theorem proceeds in several steps. The first step is crucial: we show that as a semigroup $S$ is a direct product $S= E\times G$ of a left zero semigroup $E$ (meaning that
$ef=e$ for every $e,f \in E$) and a group $G$; see Theorem~\ref{thm:leftgroup}.
Next in Proposition~\ref{bijective} we prove that
each map $\theta_{x}$ for $ x\in S$, is a permutation of $S$.
This allows us to exploit group-theoretic tools, which turn out to be central in the formulation and proof of our main results. 

In Section~\ref{sect3}, we introduce  a notion of retract $\Ret (S,s)$ (see Proposition~\ref{pro:retract}) defined via a congruence on the semigroup $(S,\cdot)$. This
retract again carries the structure of a set-theoretic solution to the Pentagon Equation, which we denote by
 $\Ret (S,s) = (\overline{S},\overline{s})$. Then using the bijectivity of $s$ we discover a dual bijective set-theoretic solution $(S,t)$, where $t=\tau s^{-1}\tau$, that defines another
semigroup structure $(S, \circ)$ on the set $S$. This yields a decomposition into a Cartesian product $S=X\times A\times G$
such that the underlying set of the retract $\overline{S}$ can be identified with $A\times G$ and
the set of idempotents $E$ of $S$ can be identified with $X\times A$.
Later, we show that the retract operator is idempotent, in other words $\Ret(\Ret(S,s))=\Ret(S,s)$, for any finite bijective set-theoretic solution to the Pentagon Equation $(S,s)$.
We emphasize that our notion of retract differs from the retraction for set-theoretic solutions to the Yang--Baxter Equation introduced in \cite{ESS}. In the Yang--Baxter setting, retraction is often iterated and serves as a measure of the complexity of the solution, whereas in the present setting the retract is a key structural ingredient and, in the finite bijective case, it allows one to reconstruct $(S,s)$ from $\Ret(S,s)$.
In Section~\ref{sect_matched} we focus on irretractable solutions $(S,s) =\Ret(S,s)=(\bar{S},\bar{s})$, so $S= A\times G$. Here we show that for an irretractable solution the dual solution $(S,t)$ endows $A$ with a group structure and  surprisingly switches the
roles of $A$ and $G$. This unravels a hidden
symmetry in the retract solution $(\overline{S},\overline{s})$ and leads to the construction of a matched
product of groups $A$ and $G$ in Theorem~\ref{thm:sol2mathc}. Conversely, every matched product of groups leads, in a natural way, to an irretractable solution to the Pentagon Equation, and it turns out to be responsible for the formula describing the set-theoretic solution
$(\overline{S},\overline{s})$, see Theorem~\ref{thm:matched}.

Finally, we show that the
original set-theoretic solution $(S,s)$ is a natural extension (lift) of its retract, so that the formula \eqref{eq:star} follows, see Theorem~\ref{thm:summary}.

Although every matched pair of groups leads
in a natural way to a set-theoretic solution to the Pentagon Equation, in Section~\ref{comments} we provide  certain concrete constructions that may be used to illustrate the main ideas and strategy used in the paper. 
In particular, in the finite case the description of all involutive set-theoretic solutions, obtained earlier in
\cite{MR4170296}, via a different approach, follows easily from
Theorem~\ref{main}.
Matched pairs of groups also appear naturally in the theory of set-theoretic solutions to the Yang--Baxter Equation, see \cite{MR2383056,MR2024436,Lu2000}. Moreover, braces and skew braces introduced by Rump in \cite{zbMATH05118810} and by Guarnieri--Vendramin in \cite{GuVe2017} in the context of set-theoretic solutions to the Yang--Baxter Equation, canonically determine a matched pair of groups (see \cite[Lemma~5.9]{MR3763907}). In the last section we show explicitly how skew braces also define irretractable set-theoretic solutions to the Pentagon Equation. Thus, our results open a new connection (at the set-theoretic level) between solutions to the Pentagon Equation and solutions to the Yang--Baxter Equation. 
Skew braces can also be described in terms of regular subgroups of holomorphs \cite{GuVe2017}. This provides a computationally efficient way of defining matched pairs of groups and creates also a bridge with Hopf-Galois structures \cite{zbMATH01014452}. For more details we refer to Section~\ref{comments}.
We conclude by relating set-theoretic solutions to the
Pentagon Equation to Hopf algebras with positive bases. In particular, we recall
that any finite set-theoretic solution to the Pentagon Equation gives rise to a Hopf algebra with a positive basis \cite{CJ26x}, and we explain why Hopf-algebraic classification results
do not directly recover the set-theoretic classification we obtained in Theorem~\ref{main}.

\section{Set-theoretic bijective solutions to the Pentagon Equation}
In this section we prove key results on the semigroup structures arising from finite set-theoretic bijective solutions to the Pentagon Equation.

\begin{defn}
    A pair $(S,s)$, where $S$ is a set and $r\colon S\times S\to S\times S$ is 
    said to be a \emph{set-theoretic solution to the Pentagon Equation} if
    the equality 
    \begin{equation}
        \label{eq:PE}    
        s_{23}s_{13}s_{12}=s_{12}s_{23}
    \end{equation}
    holds in $S\times S\times S$. 
\end{defn}

From now on, by \emph{a solution to the PE} we always mean a set-theoretic solution to the Pentagon Equation.

\begin{defn}\label{def:iso}
    Let $(S,s)$ and $(T,t)$ be solutions to the PE. We say that a map $f\colon S\to T$ 
    is a \emph{homomorphism of solutions} if 
    the diagram
    \[
    \begin{tikzcd}
    	{S\times S} & {T\times T} \\
    	{S\times S} & {T\times T}
    	\arrow["{f\times f}", from=1-1, to=1-2]
    	\arrow["s"', from=1-1, to=2-1]
    	\arrow["{f\times f}"', from=2-1, to=2-2]
    	\arrow["t", from=1-2, to=2-2]
    \end{tikzcd}
    \]
    commutes.
    The solutions $(S,s)$, $(T,t)$ are \emph{isomorphic} if there exists a bijective homomorphism of solutions $f:S\to T$.
\end{defn}

By convention, if $(S,s)$ is a solution to the PE, we write 
$s(x,y)=(xy, \theta_x(y))$, for a map 
$S\times S\to S$, $(x,y)\mapsto x\cdot y$, and maps 
$\theta_x\colon S\to S$ for $x,y\in S$. 

The following
lemma appears in \cite{MR1637789}. 

\begin{lem}
\label{lem:Kashaev}
    Let $S$ be a set and $s\colon S\times S\to S\times S$ be a map. Then 
    $(S,s)$ is a set-theoretic solution to the PE if and only if
    \begin{align}
        \label{PE1}x(yz)&=(xy)z,\\
        \label{PE2}\theta_x(y)\theta_{xy}(z)&=\theta_x(yz),\\
        \label{PE3}\theta_{\theta_x(y)}\theta_{xy}(z)&=\theta_y(z),
    \end{align}
    for all $x,y,z\in S$. 
    Moreover, $s$ is bijective if and only if for every $x,y\in S$ there exists a unique pair $(u,v)\in S\times S$
    such that $uv=x$ and $\theta_u(v)=y$. 
\end{lem}

\begin{proof}
    It is straightforward. 
\end{proof}

The first crucial step towards the classification will be to describe the structure of the semigroup $(S,\cdot)$. Namely, we will prove that if $(S,s)$ is a finite bijective solution to the PE then $S$ is a left group, i.e. $S$ is isomorphic to $E\times G$ where $E$ is a set, $G$ is a group, and the multiplication on $S$ is defined by 
\begin{align*}
    (e,g)(f,h)= (e,gh)
\end{align*}
for every $(e,g),(f,h)\in E\times G$.
We refer the reader to \cite{MR0132791} for the necessary background on semigroups. 

For any semigroup $S$ we denote with $E(S)$ the set of idempotent elements of $S$.

\begin{thm}\label{thm:leftgroup}
    Let $(S,s)$ be a finite bijective solution to the PE. Then $S$ is a left group.
\end{thm}

\begin{proof}
    By \cite[Proposition 2.2]{MR4170296}, if $(S,s)$ is a solution to the PE of finite order, then $S$ is a simple semigroup. This implies, in particular, that  if $(S,s)$ is bijective and $S$ is finite, $S$ is a completely simple semigroup (cf \cite[page 76]{MR0132791}).

    By Rees' theorem it follows that we may identify $S$ with a completely simple matrix semigroup $\mathcal{M}(G,I,\Lambda;P)$ with $G$ a finite group whenever $(S,s)$ is a finite bijective solution to the PE. We recall that, if $G$ is a group, $I$ and $\Lambda$ are non-empty sets, and $P$ is a matrix indexed by $I$ and $\Lambda$ with entries $p_{\lambda,i} \in G$ then the \emph{Rees matrix semigroup} $S=\mathcal{M}(G,I,\Lambda;P)$ is the set $I\times G \times \Lambda$ with multiplication given by
\begin{align*}
    (i, g, \lambda)(j, h, \mu) = (i, g p_{\lambda,j} h, \mu).
\end{align*}
In particular, the set of idempotent elements of $S$ is given by
\begin{align*}
    E(S)=\left\{(i,p_{\lambda i}^{-1},\lambda)\colon  i\in I, \lambda\in \Lambda \right\}.
\end{align*}

If $e=(i,p_{\lambda i}^{-1},\lambda)\in E(S)$, then the right ideal 
generated by $e$ is 
\[
eS =\{(i, g,\mu) \colon g\in G, \mu\in\Lambda\}
\]
and the left ideal generated by $e$ is 
\[
Se =\{(j, g,\lambda) \colon j \in I, g\in G\}.
\]
Let $m=|I|$ and $n=|\Lambda|$. Then  
$|eS| = |G||\Lambda| = n|G|$ and $|Se| = |I||G| = m|G|$. Now notice that for every $x=(i, g, \mu) \in eS$, there exist exactly $m$ elements $y\in Se$  such that $xy = e$, i.e. $y \in \{(j, p_{\mu,j}^{-1}g^{-1}p_{\lambda,i}^{-1},
\lambda) \colon j \in I\}$. Therefore we get $nm|G|$ different pairs $(x, y) \in S \times S$ such that $xy = e$.
Next, we use property \eqref{PE2} with $xy=e$, and then $y=ye$ and $x=ex$. So, in particular
\begin{align*}
    \theta_x(y) = \theta_x(ye) = \theta_x(y)\theta_{xy}(e) = \theta_x(y)\theta_e(e),
\end{align*}
and  $\theta_x(y)\mathcal{L}\theta_e(e)$, where as usual $\mathcal{L}$ denotes the Green's relation. It follows that $|\{\theta_x(y) : xy = e\}|\leq |S\theta_e(e)|= m|G|$. However, $|\{(x,y) \colon xy = e\}| = mn|G|$. Since $s$ is bijective, considering the set 
\begin{align*}
    s^{-1}(\{(xy,\theta_x(y)) : xy=e\}),
\end{align*}
we get that $mn|G|\leq m|G|$. Hence $n=1$ and the statement follows.
\end{proof}

The second crucial step is to describe important properties of the maps $\theta_x$, for $x\in S$. In particular, we will prove that all of them are bijective. We begin with some preparatory lemmas.

\begin{lem}\label{lem:idempotent_reduction}
Let $(S,s)$ be a finite bijective solution to the PE and set $E=E(S)$. Then
$\theta_x(E)\subseteq E$ for all $x\in S$. In particular, $s$ restricts to a
finite bijective solution on $E$ (since $E(S)$ is closed under multiplication).

Moreover, if $E(S)=S$ (equivalently $S$ is a left zero semigroup), then for
each $x\in S$ the map $\theta_x$ is bijective and is either the identity or a
fixed-point free permutation.
\end{lem}
\begin{proof}
    By Theorem~\ref{thm:leftgroup}, $S$ is a left group, so 
    every element in $E(S)$ is a right identity, i.e. $xe=x$ for all $x\in S$ and $e\in E(S)$. Let $x\in S$ and $e\in E(S)$. From \eqref{PE2} it follows that 
    \begin{align*}
        \theta_x(e) = \theta_x(e^2) = \theta_x(e)\theta_{xe}(e) 
        =\theta_x(e)\theta_x(e),
    \end{align*}
    i.e. $\theta_x(e)\in E(S)$.

    For the second statement see \cite[Lemma 2.4]{MR4170296}.
\end{proof}

\begin{lem} \label{right}
    Let $(S,s)$ be a finite bijective solution to the PE and let $e\in E(S)$. If $x\in eS$, then $\theta_y(x)\in \theta_y(e)S$.
    \begin{proof}
        By Lemma~\ref{lem:idempotent_reduction}, $\theta_y(e)\in E(S)$. Let $e\in E(S)$, $x \in eS$ and $y\in S$. By \eqref{PE2} and from $x=ex$ we get
        \begin{align*}
            \theta_y(x) = \theta_y(ex) = \theta_y(e)\theta_{ye}(x)
            = \theta_y(e)\theta_{y}(x),
        \end{align*}
        hence $\theta_y(x) \in \theta_y(e)S$.
    \end{proof}
\end{lem}

For a map $\theta\colon S \rightarrow S$ let 
$\sim_{\theta}$ be the relation $\{(a,b)\in S\times S\mid \theta (a) = \theta (b)\}$.

\begin{lem}\label{lem:ker}
Let $(S,s)$ be a finite bijective solution to the PE. 
Then the following statements hold: 
\begin{enumerate}
    \item $\sim_{\theta_x}\ =\ \sim_{\theta_{y}}$, for any $x,y \in S$.
    \item $|\theta_x(S)|=|\theta_y(S)|$, for any $x,y\in S$.
    \item The relation $\sim \ =\ \sim_{\theta_x}$, for any $x$, is a left congruence on $S$.
\end{enumerate}
\begin{proof}
        Let us prove (1). Let $x,y\in S$. By equation \eqref{PE3}, 
    \begin{align}\label{star}
        \sim_{\theta_{xy}}\ \subseteq\ \sim_{\theta_y}.
    \end{align}
    Moreover, since $S$ is a left group, there exists $v$ such that $y = vxy$. It follows that $\sim_{\theta_y}\ =\ \sim_{\theta_{vxy}}\overset{\eqref{star}}{\subseteq} \sim_{\theta_{xy}} \overset{\eqref{star}}{\subseteq} \sim_{\theta_y}$. Therefore, $\sim_{\theta_{xy}}\ =\ \sim_{\theta_y}$. Finally, let $x'\in S$. Since $S$ is left simple, there exists $z$ such that $x'=zx$. Then $\sim_{\theta_{x'}}\ =\ \sim_{\theta_{zx}}\ =\ \sim_{\theta_{x}}$.
    
    We prove (2). Since $\sim_{\theta_x}$ is an equivalence relation, by (1) and the fact that $S$ is finite it follows that $|\theta_x(S)|=|\theta_{x'}(S)|$, for any $x,x'\in S$.
    
    We now prove (3). By (1), $\sim_{\theta_x}\ =\ \sim_{\theta_{y}}$, for every $x,y\in S$. We denote this relation simply with $\sim$.
    Let $a\sim b$ and $y\in S$. Then $\theta_{xy}(a) = \theta_{xy}(b)$ for any $x \in S$. It follows that $\theta_x(ya) \overset{\eqref{PE2}}{=}\theta_{x}(y)\theta_{xy}(a)=\theta_x(y)\theta_{xy}(b)\overset{\eqref{PE2}}{=}\theta_x(yb)$. Therefore $\sim$ is a left congruence.
\end{proof}
\end{lem}

\begin{lem}\label{lem:subgroup}
    Let $(S,s)$ be a finite bijective solution to the PE. Let $e\in E(S)$.
    Then, there exists a subgroup $H_e$ of $G_e = eS$, such that for any $a,b\in G_e$,
\begin{align*}
    a\sim b \quad \Longleftrightarrow \quad aH_e=bH_e.
\end{align*}

\begin{proof}
   Define 
   \begin{align*}
       H_e=\left\{ a \in G_e \ |\  a\sim e\right\}.
   \end{align*}
   We claim that $H_e$ is a subgroup of $G_e$. Clearly $e\in H_e$. If $a,b\in H_e$, then for any $x\in S$ we have
   \begin{align}
       \label{eqa} \theta_x(a) = \theta_x(e)\\
       \label{eqb} \theta_x(b) = \theta_x(e).
   \end{align}
   Then 
   \begin{align}
       \theta_x(ab)\overset{\eqref{PE2}}{=} \theta_x(a)\theta_{xa}(b) \overset{\eqref{eqa}\eqref{eqb}}{=}\theta_x(e)\theta_{xa}(e) = \theta_x(e)
   \end{align}
   where the last equality holds since, by Lemma~\ref{lem:idempotent_reduction}, $\theta_x(e),\theta_{xa}(e)\in E(S)$ and any idempotent element in $S$ is a right identity.

   Now, let $a\in H_e$ and denote by $\bar{a}$ the unique element in $G_e$ such that $a\bar{a} = e = \bar{a}a$. It follows that, for any $x\in S$
   \begin{align*}
       \theta_x(e) = \theta_x(\bar{a}a) = \theta_x(\bar{a}) \theta_{x\bar{a}}(a) = \theta_x(\bar{a}) \theta_{x\bar{a}}(e) = \theta_x(\bar{a}),
   \end{align*}
   where the last equality holds because $\theta_{x\bar{a}}(e)\in E(S)$ and so it is a right identity. It follows that $\bar{a}\in H_e$, so $H_e$ is a subgroup of $G_e$. The assertion follows from Lemma~\ref{lem:ker}.
\end{proof}
\end{lem}

We are now ready to derive a key structural result that will clearly indicate that group theoretical methods must naturally appear when dealing with finite bijective solutions to the PE.

\begin{pro} \label{bijective}
    If $(S,s)$ is a finite bijective solution to the PE, then
    $\theta_{x}$ is bijective for every $x\in S$. 
\end{pro} 

\begin{proof}
    Since $(S,s)$ is a finite bijective solution to the PE, by Theorem~\ref{thm:leftgroup} $S$ is a left group, i.e.,  
    $S=E\times G$, where $E=E(S)$ and $G$ is a group. We write $G_e=eG$ for $e\in E$. By Lemma~\ref{lem:ker}(1), if $x,y\in S$ then $\theta_{\theta_{x}(y)}$, $\theta_{xy}$ and $ \theta_y$ determine the same congruence $\sim$. 
    Moreover by \eqref{PE3} we have the inclusion of images $ \theta_y(S) \subseteq \theta_{\theta_{x}(y)}(S)$. 
    By Lemma~\ref{lem:ker}(2), since $S$ is finite, we get $ \theta_y(S) = \theta_{\theta_{x}(y)}(S)$ and also $|\theta_{y}(S)| = |\theta_{xy}(S)|$. 
    Applying Equation~\eqref{PE3} yields that the restriction of $\theta_{\theta_{x}(y)}$ to $\theta_{xy}(S)$ is bijective.

    On the other hand, the relation $\sim\ =\ \sim_{\theta_x}$, for every $x\in S$, is a left congruence on $S$ by Lemma~\ref{lem:ker}(3). 
    And by Lemma~\ref{lem:subgroup}, for any $e\in E$ we can define a group $H_e=\left\{ a \in G_e \ |\  a\sim e \right\}$ such 
    that $a,b\in G_e$ and $a\sim b$ if and only if $aH_e=bH_e$. Moreover $fH_e = H_f$ for every $e,f\in E$.
    Now, since $\theta_{\theta_{x}(y)}$ is bijective when restricted to $\theta_{xy}(S)$, we get that for every $e\in E(S)$
    \begin{equation} \label{ineq}
        |\theta_{xy}(S)\cap H_e|\leq 1
    \end{equation}
But using Lemma~\ref{lem:idempotent_reduction} we have $\theta_{xy}(f)S=G_{\theta_{xy}(f)}$ and $e=\theta_{xy}(f)\in \theta_{xy}(G_f)$ is the identity of the group $G_{\theta_{xy}(f)}$. Therefore  
$\theta_{xy}(S)\cap H_e = \{ \theta_{xy}(f) \}$.
We know also that every $\theta_{e'}$, $e'\in E(S)$, permutes the set $E(S)$ (by Lemma~\ref{lem:idempotent_reduction}) and $\theta_{e'}$ maps every $G_f$ into some $G_{f'}$. 
Suppose that $\theta_z, z\in S$, are not bijections.
Then it follows that (every) $H_e$ must be a nontrivial subgroup in $G_e$. Therefore \eqref{ineq} implies that every nonidentity element $h$ of the group $H_{\theta_{xy}(f)}$ is not in the image of $\theta_{xy}$, for all $x,y\in S$. So $h\notin \bigcup_{z\in S}\theta_{z}(S)$. However,  bijectivity of the solution $s$ implies that $S= \bigcup_{z\in S}\theta_{z}(S)$. This contradiction proves that every $\theta_y$ must be bijective.
\end{proof}

Our final results in this section highlight the role of the maps $\theta_e$ determined by the idempotents $e$ of the semigroup $S$. 
\begin{pro}\label{pro:subgroupandonlyidempotent}
    Let $(S,s)$ be a finite bijective solution to the PE. Then, for every $x \in S$ there exists an $e \in E(S)$ such that $\theta_x = \theta_e$. Moreover, the set $\theta_E=\left\lbrace \theta_e \mid e \in E(S)\right\rbrace$ is a group with identity $\theta_{\theta_{e}(e)}$ independent of the choice of $e\in E(S)$.
\end{pro}
\begin{proof}
We start by proving the latter claim. Let $e\in E(S)$. Then, by relation \eqref{PE3} $$ \theta_{\theta_e(e)}\theta_{ee} = \theta_e,$$ By Proposition~\ref{bijective}, $\theta_e$ is bijective, which shows that 
$\theta_{\theta_{e}(e)}=\id_S$. Let $e,f \in E(S)$. Then, by Lemma~\ref{lem:idempotent_reduction}), 
$E(S)$ is closed under $\theta_e$, thus there is a $f_2 \in E(S)$ with $\theta_e(f_2)=f$. Hence, as $ef_2=e$, we get $$ \theta_f\theta_e=\theta_{\theta_e(f_2)}\theta_{ef_2}
\overset{\eqref{PE3}}{=}
\theta_{f_2},$$ which shows that the set $\theta_E$ is a submonoid of $\Sym(S)$. As the latter is finite, $\theta_E$ is a subgroup.

Let $x \in S$ and let $e \in E(S)$. Then by \eqref{PE3}
$$ \theta_{\theta_x(e)}\theta_{xe} = \theta_e,$$ which is equivalent to $$ \theta_x = \theta_{\theta_{x}(e)}^{-1}\theta_e.$$ As $\theta_x(e) \in E(S)$, the latter can be rewritten as $\theta_f$ for some $f \in E(S)$. 
\end{proof}

From now on we fix $e_0\in E(S,\cdot)$ and denote by $1=\theta_{e_0}(e_0)$.

\begin{pro}\label{pro:thetagidentity}
    Let $(S,s)$ be a finite bijective solution to the PE. Denote $G= G_1=1S$. Then
    for any $g \in G$, we have $\theta_g=\id_S$.
\end{pro}
\begin{proof}
    By Proposition~\ref{pro:subgroupandonlyidempotent}, $\theta_1=\id_S$. Then, 
    $$ \theta_{g}\theta_g= \theta_{\theta_1(g)}\theta_{1g}
   \overset{\eqref{PE3}}{=} \theta_g, $$ 
   which shows that $\theta_g=\id_S$. 
\end{proof}

\section{The retract of a bijective solution}  \label{sect3}
Since we assume that the solution $(S,s)$ is bijective, the inverse map $s^{-1}:S\times S \rightarrow S\times S$ can be exploited. As shown below, this leads to a dual solution of the PE and the corresponding dual semigroup structure $(S,\circ)$. In this section we combine the structure of the left groups $(S,\cdot)$ and $(S,\circ)$. In particular, we show that the set of idempotent elements $e\in (S,\cdot)$ with $\theta_e=\id_S$ coincides with the intersection of the corresponding subsemigroups of idempotents. This allows us to define the retract of a bijective solution to the PE. In particular, it identifies this intersection into a single element and produces an irretractable solution, i.e. a solution with exactly one idempotent $e \in E(S,\cdot)$ such that $\theta_e=\id_S$. While the retract of $(S,s)$ has a substantially simpler structure than $(S,s)$, Theorem~\ref{thm:summary} shows that nevertheless it retains enough information to completely recover $(S,s)$ as a natural extension.
This is in sharp contrast with the Yang--Baxter setting, where the situation is
more intricate and the retraction introduced in \cite{ESS} can be applied to a solution multiple times, yielding the multipermutation level as an invariant measure of the complexity of the solution. Even though the retraction and the multipermutation level of a solution to the Yang--Baxter Equation have been widely studied and used,
see for instance \cite{MR3815290,CJO2010,GI2018,zbMATH06006468}, they do not provide a 
mechanism for reconstructing the original solution.

Let $(S,s)$ be a bijective solution to the PE. Then we denote $$s^{-1}(x,y)=(\psi_y(x),y \circ x).$$ Then, as $s$ is a solution to the PE, we have that $$ s^{-1}_{12}s^{-1}_{13}s^{-1}_{23}=s^{-1}_{23}s^{-1}_{12}.$$

As in Lemma~\ref{lem:Kashaev}, this translates into the following equalities.
\begin{align}
        \label{PE1Inverse}z\circ (y \circ x)&=(z \circ y)\circ x,\\
        \label{PE2Inverse}\psi_z(y)\circ\psi_{z\circ y}(x)&=\psi_z(y\circ x),\\
        \label{PE3Inverse}\psi_{\psi_z(y)}\psi_{z\circ y}&=\psi_y,
    \end{align}
which shows that $t(x,y)=(x\circ y, \psi_x(y))$ is a bijective solution to PE. Hence, one obtains the following proposition on the structure of $(S,\circ)$.

\begin{pro}\label{prop:Scirc}
    Let $(S,s)$ be a finite bijective solution to the PE. Then $(S,\circ)$ is a left group.
\end{pro}

Furthermore, we note consequences that 
follow from the equalities 
$$s \circ s^{-1}=\id_{S\times S} = s^{-1} \circ s.$$
Namely, 
\begin{align}
        \label{Interplay1}\psi_y(x) (y\circ x) &= x,\\
        \label{Interplay2}\theta_{\psi_y(x)}(y\circ x) &= y,\\
        \label{Interplay3}\psi_{\theta_x(y)}(xy)&=x,\\
        \label{Interplay4}\theta_x(y)\circ (xy)&=y.
    \end{align}
Define 
\begin{align*}
    \theta:S\to \Sym(S), x\mapsto \theta_x,\qquad \text{and}\qquad \psi: S\to \Sym(S), x\mapsto \psi_x.
\end{align*}
Hence, it follows easily that $\theta$, respectively $\psi$, is a homomorphism from $(S,\circ)$, respectively $(S,\cdot)$, to $\Sym(S)$.
Denote 
    $$\ker \theta = \left\lbrace s \in S \mid \theta_s = \id_S \right\rbrace,$$ and similarly $$ \ker \psi = \left\lbrace s \in S \mid \psi_s = \id_S \right\rbrace.$$
    
\begin{lem}\label{lem:EF}
    Denoting $E= E(S,\cdot)$ and $F=E(S,\circ)$, we get $$ \ker \theta \cap E = \left\lbrace \theta_e(e) \mid e \in E\right\rbrace = \ker \psi \cap F = \left\lbrace \psi_f(f) \mid f \in F \right\rbrace.$$
\end{lem}
\begin{proof}
    We present the proof that the first three sets are equal; then the last equality is true by symmetry. As noted before, $\theta_{\theta_e(e)}=\id_S$ for all $e \in E$.  Moreover, if $e \in \ker \theta \cap E$, then $\theta_e(e)=e$, which shows the first equality.

    Let $e \in \ker \theta \cap E$. Then, by relation \eqref{Interplay4} $$ e\circ e = \theta_e(e) \circ (ee) = e,$$ which shows that $e$ is an idempotent in $(S,\circ)$. Hence, $e \in F$.
    
    Let $x \in S$. Then by \eqref{Interplay1} we find that $$ \psi_e(x) (e \circ x) = x.$$ Moreover, rewriting relation \eqref{Interplay2}, one finds, as $e \in E$, that $$e\circ x = \theta^{-1}_{\psi_e(x)}(e) \in E$$ 
    by Lemma~\ref{lem:idempotent_reduction}. Combining this information, we get that $$ x = \psi_e(x)(e \circ x) = \psi_e(x),$$ which shows that $\psi_e=\id_S$. Thus, $\ker \theta \cap E  \subseteq \ker \psi \cap F$, the reverse equality is analogous.
\end{proof}

The previous allows us to fully determine $F=E(S,\circ)$ and by symmetry $E=E(S,\cdot)$. Recall that $1=\theta_{e_{0}}(e_0)$, as fixed before Proposition~\ref{pro:thetagidentity}.

\begin{pro}\label{pro:idempotentsubsemi}
    Let $(S,s)$ be a finite bijective solution to the PE. Then, one has that $F = (\ker \theta \cap E) G_1$. In particular, $F$ is a subsemigroup of $(S,\cdot)$.
\end{pro}
\begin{proof}
    Let $e \in \ker \theta\cap E$ and $g \in G_1$. Then, by relation \eqref{Interplay4}, 
    $$ eg \circ eg = \theta_e(eg) \circ (e eg ) = eg,$$ 
    which shows that $eg \in F$. Hence, $(\ker \theta \cap E)G_1 \subseteq F.$ In particular, as $1=\theta_{\bar{e}}(\bar{e}) \in \ker \theta \cap E$ by Proposition~\ref{pro:subgroupandonlyidempotent}, one obtains that $G_1=1G_1 \subseteq F$. Vice versa, one notes that for any $f_1,f_2 \in F$ one has, by rewriting relation \eqref{Interplay3}, that 
    $$ f_1f_2 = \psi^{-1}_{\theta_{f_1}(f_2)}(f_1),$$ 
    where the latter is in $F$ by  Lemma~\ref{lem:idempotent_reduction}. This shows that $F$ is a subsemigroup 
    of $(S,\cdot)$.
    
    Let $f \in F$. Then there exist $e \in E$ and $g \in G_1$ such that $f=eg$. It remains to show that $e \in \ker \theta$. 
    As $G_1 \subseteq F$ and $F$ is a subsemigroup of $(S,\cdot)$, it follows that $e=(eg)(g^{-1}) \in E \cap F$, where $g^{-1}$ denotes the inverse of $g$ in $G_1$. 
    Let $y \in S$. Considering relation \eqref{Interplay2}, one finds that 
    $$ \theta_{e}(y \circ \psi_{y}^{-1}(e))=y.$$
    As $e \in F$, one has, by applying  Lemma~\ref{lem:idempotent_reduction} on the solution $t$ (introduced above Proposition~\ref{prop:Scirc}), that $ \psi_{y}^{-1}(e) \in F$. Thus, the above equality can be reduced to 
    $$ \theta_e(y)=y.$$  
    Hence, $e\in \ker \theta$, showing that $\theta_e=\id_S$ and proving the result.
\end{proof}

\begin{pro} \label{pro:retract}
    Let $(S,s)$ be a finite bijective solution to the PE. For any elements $e,f\in E=E(S,\cdot)$ and $g,h \in G_1$, define the following equivalence relation on $S$: $eg\approx fh$, if $\theta_e = \theta_f$ and $g=h$. Denote $\overline{S}=S/\approx$. Then $s$ induces a bijective solution $(\overline{S},\overline{s})$, which we call the retract of $(S,s)$.
\end{pro}
\begin{proof}
    Clearly, $\approx$ is a congruence on the semigroup $(S,\cdot)$. Let $eg\approx fh$ such that $e,f \in E$ and $g,h \in G_1$. Then, $\theta_e=\theta_f$ and $g=h$. Let $q\in S$. First, we claim that $\theta_q$ is independent of the choice of the representing element in the congruence class of $q$.  
    Indeed, using  relation \eqref{PE3} twice, we find that 
    $$ \theta_{eg} = \theta^{-1}_{\theta_e(g)}\theta_g= \theta^{-1}_{\theta_f(g)}\theta_g=\theta_{fg},$$ 
    which shows the claim. 
    Secondly, using relation \eqref{PE2} and that $e\in E$, one gets that $$ \theta_{q}(eg)=\theta_q(e)\theta_{q}(g).$$ 
    Similarly, 
    $$ \theta_q(fg)=\theta_q(f)\theta_q(g).$$  
    
    From relation \eqref{PE3} we get that $$ \theta_{\theta_q(e)} = \theta_{\theta_q(e)}\theta_{qe}\theta^{-1}_{qe}=\theta_e\theta_{q}^{-1} = \theta_f\theta_{q}^{-1}= \theta_{\theta_q(f)}.$$ 
    This means that that $\theta_q(e) \approx \theta_q(f)$. 
    So $\theta_{q}(eg)=\theta_q(e)\theta_{q}(g) \approx  \theta_q(f)\theta_{q}(g)=\theta_q(fg)$.
    Thus, indeed, $\theta_{eg}$ and $\theta_{fg}$ induce the same map on $\overline{S}$. 
   
   It now follows easily that $\overline{S}$ carries an induced semigroup structure and satisfies relations \eqref{PE1}, \eqref{PE2} and \eqref{PE3} for the maps induced by $\theta_q$ for all $q \in S$, proving the result.
\end{proof}

Let $(S,s)$ be a finite bijective solution to the PE and $(\overline{S}, \overline{s})$ its retract. We denote by 
\begin{align*}
    \pi: S \to \overline{S}, x\to \overline{x}
\end{align*}
the homomorphism from $(S,\cdot)$ to $\overline{S}$ with respect to the induced semigroup structure. 
Moreover we denote with $\overline{\theta}$ the induced $\theta$-map
for the retract. In other words 
\begin{align*}
    \bar{s}(\bar{x},\bar{y}) = (\overline{xy}, \bar{\theta}_{\bar{x}}(\bar{y}))
\end{align*}
for $x,y\in S$.

Note that from the definition of the retract it easily follows that if $e\in E(S)$ and $x\in S$ are such that $e\approx x$ then $x\in E(S)$. In particular, $\pi(E(S))=E(\overline{S})$ and we denote it by $\overline{E}$.

We will say that a solution $(S,s)$ is \emph{irretractable} if $\overline{S}=S$.

\begin{pro}    \label{pro:ret_irret}
    Let $(S,s)$ be a finite bijective solution to the PE and $(\overline{S},\overline{s})$ its retract. 
    Then, $|\ker \overline{\theta} \cap \overline{E}|=1$.  In particular, $(\overline{S},\overline{s})$ is irretractable.
\end{pro}
\begin{proof}
    Let $x \in \overline{E}$ with $\overline{\theta}_{x}=\id_{\overline{S}}$. Then there exists an $e \in E$ such that $\pi(e)=x$. Denote $f= \theta_e(e) \in \ker \theta$. As $\overline{\theta_{e}(f)}=\overline{\theta}_{\overline{e}} (\overline{f}) = \overline{f}$, it follows that $\theta_{f} =\theta_{\theta_{e}(f)}$, so that $\theta_e(f)\in \ker \theta$. However, $$ \theta_{\theta_e(f)}\theta_e=\theta_f=\id_S,$$ which implies, using that $\theta_e(f)\in \ker\theta$ that $\theta_e=\theta^{-1}_{\theta_e(f)}=\id_S$. This shows that $e \in \ker \theta$. Hence, both assertions follow by the remarks made before.
\end{proof}

\section{From matched pairs of groups to irretractable solutions}\label{sect_matched}

Recall that a quadruple $(G, A, \sigma, \delta)$ is a \emph{matched pair of groups} if $A$ and $G$ are groups and there are a left action $\sigma: A\to \Sym(G)$ and a right action $\delta: G \to \Sym(A)$ such that 
\begin{align}  
    \label{eq:sigmaprod}\sigma_a(xy) = \sigma_a(x) \sigma_{\delta_{x}(a)}(y) \\
    \label{eq:deltaprod}\delta_x(ab) = \delta_{\sigma_b(x)}(a)\delta_x(b).
\end{align}
If the quadruple $(G, A, \sigma, \delta)$  forms a matched pair of groups, then $G \times A$ is a group with multiplication
\begin{align*}
    (x,a)(y,b) = (x\sigma_a(y), \delta_y(a)b) 
\end{align*}
where $a,b \in A$ and $x,y \in G$. The inverse of $(a,x)$ is given by
\begin{align}
    (x,a)^{-1} = (\sigma_{a^{-1}}(x^{-1}), \delta_{x^{-1}}(a^{-1})).
\end{align}
This group will be denoted by $G\bowtie A$ and it is known also as a \emph{biproduct} or \emph{Zappa--Sz\'ep product} of $G$ and $A$. We refer to \cite{MR2024436} for details. 
From equalities~\eqref{eq:sigmaprod} and \eqref{eq:deltaprod}, one gets
\begin{align}
    \label{eq:sigmainv}(\sigma_a(x))^{-1} = \sigma_{\delta_x(a)}(x^{-1}),\\
    \label{eq:deltainv}(\delta_x(a))^{-1} = \delta_{\sigma_a(x)}(a^{-1}).
\end{align}

Now we show that every matched pair of groups naturally leads to an irretractable bijective solution to the PE.

\begin{thm}\label{thm:matched}
    Let $(G, A, \sigma, \delta)$ be a matched pair of groups and let $S$ be the Cartesian product $A\times G$ of $A$ and $G$. Then the map $s:S\times S \to S \times S$ defined by 
    \begin{align}
        s((a, x), (b,y)) = ((a,xy), (b(\delta_x(a))^{-1}, \sigma_{\delta_x(a)}(y)))
    \end{align}
    is an irretractable bijective solution to the PE.
\end{thm}

\begin{proof}
    Let us define on $A\times G$ the following operation
    \begin{align}
        (a,x)(b,y)= (a,xy).
    \end{align}
    Clearly, the latter endows  $A\times G$ with a semigroup structure, even a left group. Hence, \eqref{PE1} is satisfied.
    Moreover, let us also define
    \begin{align*}
        \theta_{(a,x)}(b,y) =  (b(\delta_x(a))^{-1}, \sigma_{\delta_x(a)}(y))
    \end{align*}
    for all $a,b\in A$, $x,y \in G$. Then
    \begin{align*}
        \theta_{(a,x)}(b,y)\theta_{(a,x)(b,y)}(c,z)&= (b(\delta_x(a))^{-1}, \sigma_{\delta_x(a)}(y))\theta_{(a,xy)}(c,z)\\
        &= (b(\delta_x(a))^{-1}, \sigma_{\delta_x(a)}(y)) (c(\delta_{xy}(a))^{-1}, \sigma_{\delta_{xy}(a)}(z))\\
        &= (b(\delta_x(a))^{-1}, \sigma_{\delta_x(a)}(y)\sigma_{\delta_{xy}(a)}(z))\\
        &=(b(\delta_x(a))^{-1}, \sigma_{\delta_x(a)}(y)\sigma_{\delta_{y}\delta_x(a)}(z))\\
        &\overset{\eqref{eq:sigmaprod}}{=} (b(\delta_x(a))^{-1}, \sigma_{\delta_x(a)}(yz)\\
        &=\theta_{(a,x)}(b,yz)= \theta_{(a,x)}((b,y)(c,z)).
    \end{align*}
    Hence, \eqref{PE2} is satisfied. Now, let us first compute
    \begin{align*}
        \delta_{\sigma_{\delta_x(a)}(y)}(b(\delta_x(a))^{-1})
        &=\delta_{\sigma_{(\delta_x(a))^{-1}}\sigma_{\delta_x(a)}(y)}(b)\delta_{\sigma_{\delta_x(a)}(y)}((\delta_x(a))^{-1})\\
        &\overset{\eqref{eq:deltainv}}{=}\delta_{y}(b)\delta_{\sigma_{\delta_x(a)}(y)}\delta_{\sigma_a(x)}(a^{-1})\\
        &\overset{\eqref{eq:sigmainv}}{=}\delta_{y}(b)\delta_{\sigma_a(x)\sigma_{\delta_x(a)}(y)}(a^{-1})\\
        &=\delta_{y}(b)\delta_{\sigma_a(xy)}(a^{-1})\\
        &=\delta_{y}(b)(\delta_{xy}(a))^{-1}.
    \end{align*}
    It follows that
    \begin{align*}
        &\theta_{\theta_{(a,x)}(b,y)}\theta_{(a,x)(b,y)}(c,z)
        =\theta_{(b(\delta_x(a))^{-1}, \sigma_{\delta_x(a)}(y))}
        (c(\delta_{xy}(a))^{-1}, \sigma_{\delta_{xy}(a)}(z))\\
        &=(c(\delta_{xy}(a))^{-1}(\delta_{\sigma_{\delta_x(a)}(y)}(b(\delta_x(a))^{-1}))^{-1}, \sigma_{\delta_{\sigma_{\delta_x(a)}(y)}(b(\delta_x(a))^{-1})}\sigma_{\delta_{xy}(a)}(z))\\
        &=(c(\delta_{xy}(a))^{-1}(\delta_{y}(b)(\delta_{xy}(a))^{-1})^{-1}, \sigma_{\delta_{y}(b)(\delta_{xy}(a))^{-1}}\sigma_{\delta_{xy}(a)}(z))\\
        &=(c(\delta_{y}(b)(\delta_{xy}(a))^{-1}\delta_{xy}(a))^{-1}, \sigma_{\delta_{y}(b)(\delta_{xy}(a))^{-1}\delta_{xy}(a)}(z))\\
        &=(c(\delta_{y}(b))^{-1}, \sigma_{\delta_{y}(b)}(z))\\
        &=\theta_{(b,y)}(c,z).
    \end{align*}
    Hence, \eqref{PE3} is satisfied.

    We show  that $s$ is bijective. Let $(a,x),(b,y) \in S$. Then we need to show that there exist unique $(c,u),(d,v) \in S$ such that $$s( (c,u), (d,v)) = ((a,x),(b,y)).$$ Clearly, it is necessary by the definition of $s$ that $c=a$, $uv=x$ and $\sigma_{\delta_u(a)}(v)=y$. Next, the equality $x= uv$ is equivalent via application of $\sigma_a$ in both terms, by equality \eqref{eq:sigmaprod}, to  $$\sigma_a(x) = \sigma_a(uv) = \sigma_a(u) \sigma_{\delta_u(a)}(v)= \sigma_a(u)y.$$ Hence, $$ u = \sigma_{a}^{-1}(\sigma_a(x)y^{-1}).$$ Thus, $v$ is uniquely determined as $ \sigma_{\delta_u(a)}^{-1}(y).$ Finally, $d = b \delta_u(a)$, which shows that such $(c,u),(d,v) \in S$ exist and are unique.
    
    Now we are left to prove that $(A\times G, s)$ is irretractable. 
    First note that $A\times \{1\}$ coincides with the set of idempotents of $A\times G$. Indeed 
    \begin{align*}
        (a,x)(a,x)= (a,x) \iff x=x^2 \iff x =1.
    \end{align*}
    Moreover note that if $(a,1)$ and $(b,1)$ are such that 
    \begin{align*}
        \theta_{(a,1)}(c,z) = \theta_{(b,1)}(c,z)
    \end{align*}
    for every $(c,z)$. Then it follows that 
    \begin{align*}
        &(c(\delta_1(a))^{-1}, \sigma_{\delta_1(a)}(z)) = (c(\delta_1(b))^{-1}, \sigma_{\delta_1(b)}(z))\\ 
        &\iff c(\delta_1(a))^{-1} = c(\delta_1(b))^{-1} \text{ and } \sigma_{\delta_1(a)}(z)=\sigma_{\delta_1(b)}(z) \\
        &\iff ca^{-1} = cb^{-1} \text{ and } \sigma_{a}(z)=\sigma_{b}(z)\\
        &\iff a = b.
    \end{align*}
    Hence, 
    \begin{align*}
        (a,x)\approx(b,y) \iff \theta_{(a,1)}=\theta_{(b,1)} \text{ and } x=y \iff a=b \text{ and } x=y,
    \end{align*}
    which means that $(A\times G, s)$ is irretractable.
\end{proof}

\section{From irretractable solutions to matched pairs of groups}

In this section, we prove that finite irretractable bijective solutions to the PE define matched pairs of groups. We will use again the element $1\in S$, introduced before Proposition~\ref{pro:thetagidentity} and used in the previous section and the notation $t(x,y)=(x\circ y, \psi_x(y))$ for the map $t=\tau s^{-1}\tau$ introduced in Section~\ref{sect3}. We start with a few technical lemmas. 

\begin{lem}\label{lem:AandG}
    Let $(S,s)$ be a finite irretractable bijective solution to the PE. Then, $G=1\cdot S$ is a subgroup of $(S,\cdot)$ and coincides with $E(S,\circ)$. Similarly, $A=1\circ S$ is a subgroup of $(S,\circ)$ and coincides with $E(S,\cdot)$.
\end{lem}

\begin{proof}
     By Proposition~\ref{pro:idempotentsubsemi} we have $E(S,\circ)=F= (\ker\theta \cap E(S,\cdot))G_1$. Since $(S,s)$ is irretractable, necessarily $\ker\theta\cap E(S,\cdot)=\{1\}$, hence $E(S,\circ)=G_1 =G$. By symmetry, applying the same argument to the solution $t=\tau s^{-1}\tau$, one also obtains $A=E(S,\cdot)$.
\end{proof}

\begin{lem}\label{lem:actions}
    Let $(S,s)$ be a finite irretractable bijective solution to the PE and let $A=1\circ S$ and $G=1\cdot S$. Then 
    \begin{itemize}
        \item the map $\sigma: A\to \Sym(G)$ defined by $\sigma_a(x) = 1\cdot \theta_{(a,1)}(1,x)$ is a left action of $(A,\circ)$ on the set $G$,
       \item the map $\delta: G \to \Sym(A)$ defined by $\delta_x(a) = (1 \circ \psi_{(1,x^{-1})}(a^{-1},1))^{-1}$ is a right action of $(G,\cdot)$ on the set $A$.
    \end{itemize}
\end{lem}

\begin{proof}
    Let $\sigma: A\to \Sym(G)$ be defined by $\sigma_a(x) = 1\cdot \theta_{(a,1)}(1,x)$ and let the map $\delta: G \to \Sym(A)$ be defined by $\delta_x(a) = (1 \circ \psi_{(1,x^{-1})}(a^{-1},1))^{-1}$.
    Let $a,b\in A$ and $x,y \in G$.  Since $\theta$ is a homomorphism from $(S,\circ)$ to $\Sym(S)$, it follows that 
    \begin{align*}
        \sigma_a\sigma_b(x) &= 1\cdot \theta_{(a,1)}(1,1\cdot \theta_{(b,1)}(1,x)) \\
        &= 1 \cdot \theta_{(a,1)}((1,1)\cdot\theta_{(b,1)}(1,x))\\
        &\overset{\eqref{PE2}}{=} 1 \cdot \theta_{(a,1)}(1,1)\cdot\theta_{(a,1)\cdot (1,1)}\theta_{(b,1)}(1,x)\\
        &=1 \cdot \theta_{(a,1)\circ (b,1)}(1,x)\\
        &= 1 \cdot \theta_{(a\circ b, 1)}(1,x)\\
        &=\sigma_{a\circ b}(x).
    \end{align*}
    Let us define $\sigma': G \to \Sym(A)$ by $\sigma'_x(a) = 1 \circ \psi_{(1,x)}(a,1)$ and prove that it is a left action of $(G,\cdot)$ on the set $A$. 
    Indeed,
    \begin{align*}
        \sigma'_x\sigma'_y(a) &= 1\circ \psi_{(1,x)}(1,1\circ \psi_{(1,y)}(a,1)) \\
        &= 1 \cdot \psi_{(1,x)}((1,1)\circ\psi_{(1,y)}(a,1))\\
        &\overset{\eqref{PE2Inverse}}{=} 1 \circ \psi_{(1,x)}(1,1)\circ\psi_{(1,x)\circ (1,1)}\psi_{(1,y)}(a,1)\\
        &=1 \circ \psi_{(1,x)\dot (1,y)}(a,1)\\
        &= 1 \cdot \psi_{(1,xy)}(a,1)\\
        &=\sigma'_{xy}(a).
    \end{align*}
    Finally, as $\sigma'$ is a left action and $\delta_x(a)=(\sigma'_{x^{-1}}(a^{-1}))^{-1}$, we have
    \begin{align*}
        \delta_x\delta_y(a)&=\delta_x(\sigma'_{y^{-1}}(a^{-1})^{-1})\\
        &=\sigma'_{x^{-1}}(\sigma'_{y^{-1}}(a^{-1}))^{-1}\\
        &= \sigma'_{x^{-1}y^{-1}}(a^{-1})^{-1}\\
        &=\delta_{yx}(a).
    \end{align*}
    Hence $\delta$ is a right action of $(G,\cdot)$ on the set $A$.
\end{proof}
\begin{lem}\label{lem:linkingformula}
    Let $(S,s)$ be a finite irretractable bijective solution to the PE. With the notation of Lemma~\ref{lem:actions} one has for any $a \in A$ and $x\in G$ that \begin{align}
        \label{linking1}\theta_{(a,1)}(1,x) &= (1,x) \circ \psi_{(1,x)}^{-1}((a^{-1},1))\\
        \label{linking2}\theta_{(a,1)(1,x)} &=\theta^{-1}_{\theta_{(a,1)}(1,x)}\\
        \label{linking3}\psi_{(1,x)}(a,1)&= (a,1)\cdot \theta^{-1}_{(a,1)}(1,x^{-1})\\
        \label{linking4} \psi_{(1,x)\circ (a,1)} &= \psi^{-1}_{\psi_{(1,x)}(a,1)}
    \end{align}
\end{lem}
\begin{proof}
    First, by definition of $s$ and $s^{-1}$, one obtains that $$s^{-1}(\psi^{-1}_{(1,x)}(a^{-1},1), (1,x))=((a^{-1},1),(1,x) \circ \psi_{(1,x)}^{-1}(a^{-1},1)).$$ This implies that $$s((a^{-1},1),(1,x) \circ \psi_{(1,x)}^{-1}(a^{-1},1))=(\psi^{-1}_{(1,x)}(a^{-1},1), (1,x)).$$ In particular, one finds that $\theta_{(a^{-1},1)}((1,x) \circ \psi_{(1,x)}^{-1}(a^{-1},1))=(1,x)$. As $\theta$ is an action of $(A,\circ)$ on the set $S$, this can be rewritten to $$\theta_{(a,1)}(1,x) = (1,x)\circ \psi_{(1,x)}^{-1}(a^{-1},1),$$ which shows the first equality.

    Second, by relation (\ref{PE3}), one finds that $$ \theta_{\theta_{(a,1)}(1,x)} \theta_{(a,1)(1,x)} = \theta_{(1,x)}.$$
    As $\theta_{(1,x)}=\id_S$, due to Proposition~\ref{pro:thetagidentity}, the previous can be rewritten to $$ \theta_{(a,1)(1,x)}=\theta^{-1}_{\theta_{(a,1)}(1,x)},$$ which shows the second equality.

    Third, by definition of $s$, one finds that $$ s((a,1),\theta_{(a,1)}^{-1}(1,x^{-1}))= ((a,1)\cdot \theta_{(a,1)}^{-1}(1,x^{-1}), (1,x^{-1})). $$ Applying $s^{-1}$ on both sides of the equality, one finds in particular that $$ \psi_{(1,x)^{-1}}((a,1)\cdot \theta_{(a,1)}^{-1}(1,x^{-1})). $$ As $\psi$ is a left action of $(G,\cdot)$ on $S$, this can be rewritten to $$ \psi_{(1,x)}(a,1)= (a,1) \cdot \theta_{(a,1)}^{-1}(1,x),$$ which shows the third equality.

    Finally, by relation (\ref{PE3Inverse}), it follows that $$ \psi_{\psi_{(1,x)}(a,1)}\psi_{(1,x)\circ (a,1)} = \psi_{(a,1)}.$$ As the latter, by a symmetric result to Proposition~\ref{pro:thetagidentity}, is the identity on $S$, one finds that $$ \psi_{(1,x)\circ (a,1)} = \psi^{-1}_{\psi_{(1,x)}(a,1)}, $$ which shows the result.
\end{proof}

We are now ready for the main result of this section.

\begin{thm}\label{thm:sol2mathc}
    Let $(S,s)$ be a finite irretractable bijective solution to the PE. Then $A$ and $G$ form a matched pair of groups with respect to $\sigma$ and $\delta$ defined in Lemma~\ref{lem:actions} and $(S,s)$ is isomorphic to the solution associated to $(A,G,\sigma, \delta)$ as in Theorem~\ref{thm:matched}.
\end{thm}
\begin{proof}
    By Lemma~\ref{lem:actions} one has a left action $\sigma: A \rightarrow \Sym(G)$ and a right action $ \delta: G \rightarrow \Sym(A)$. To show that the quadruple $(G,A,\sigma,\delta)$ is a matched pair of groups, it remains to show that for any $x,y \in G$ and $a,b \in A$ one has that 
    \begin{align*}
    \sigma_a(xy) = \sigma_a(x) \sigma_{\delta_{x}(a)}(y)\\
    \delta_x(ab) = \delta_{\sigma_b(x)}(a)\delta_x(b).
\end{align*}
We start with the former. Let $a\in A$ and $x,y \in G$.
\begin{align*}
\sigma_a(xy)&= 1\cdot \theta_{(a,1)}(1,xy)\\ 
&= 1\theta_{(a,1)}((1,x)(1,y))\\
&\overset{\eqref{PE2}}{=} 1\theta_{(a,1)}(1,x)\theta_{(a,1)(1,x)}(1,y)\\
&\overset{\eqref{linking1}}{=}\sigma_a(x)\cdot \theta^{-1}_{\theta_{(a,1)}(1,x)}(1,y)\\
&\overset{\eqref{linking2}}{=}\sigma_a(x) \cdot \theta^{-1}_{(1,x)\circ \psi^{-1}_{(1,x)}(a^{-1},1)}(1,y).
\end{align*}
As $\theta$ is a left action of $(S,\circ)$ on the set $S$ and $\theta_{(1,y)}=\id_S$ for all $y \in G$, one obtains that \begin{equation}\label{eq:recalibrate} \theta^{-1}_{(1,x)\circ \psi^{-1}_{(1,x)}(a^{-1},1)} = \theta^{-1}_{(1,1)\circ \psi^{-1}_{(1,x)}(a^{-1},1)}. \end{equation} Hence, reprising the earlier calculation, \begin{align*}
    \sigma_a(xy)&=\sigma_a(x) \cdot \theta^{-1}_{(1,x)\circ \psi^{-1}_{(1,x)}(a^{-1},1)}(1,y)\\
    &= \sigma_a(x) \cdot \theta^{-1}_{(1,1)\circ \psi^{-1}_{(1,x)}(a^{-1},1)}(1,y)\\
    &= \sigma_a(x) \cdot \theta_{(\delta_x(a),1)}(1,y),
    \end{align*}
where we used the fact that $\psi$ is an action of $G$ on $S$ and $\theta$ is an action of $A$ on $S$. As $(1,1)$ is a right identity of $(S,\cdot)$, one can continue the chain of equalities as follows,
\begin{align*}
   \sigma_a(xy) &= \sigma_a(x) \cdot \theta_{(\delta_x(a),1)}(1,y)\\
    &= \sigma_a(x) \cdot (1,1) \cdot \theta_{(\delta_x(a),1)}(1,y)\\
    &= \sigma_a(x) \cdot \sigma_{\delta_x(a)}(y).
\end{align*}
Finally, it remains to show that $$ \delta_x(ab) = \delta_{\sigma_b(x)}(a)\delta_x(b).$$
We have 
\begin{align*}
    \delta_x(a\circ b) &= \sigma'_{x^{-1}}(b^{-1}\circ a^{-1})^{-1}\\
    &=\left((1,1)\circ\psi_{(1,x^{-1})}(b^{-1} \circ a^{-1},1) \right)^{-1}\\
    &\overset{\eqref{PE2Inverse}}{=}\left( (1,1) \circ \psi_{(1,x^{-1})}(b^{-1},1) \circ \psi_{(1,x^{-1}) \circ (b^{-1},1)}(a^{-1},1)\right)^{-1}\\
    &=\left( (1,1) \circ \psi_{(1,x^{-1})}(b^{-1},1) \circ (1,1)\circ \psi_{((1,x^{-1}) \circ (b^{-1},1)}(a^{-1},1)\right)^{-1}\\
    &= \left( (1,1)\circ \psi_{((1,x^{-1}) \circ (b^{-1},1)}(a^{-1},1)\right)^{-1} \circ \delta_x(b)\\
    &\overset{\eqref{linking4}}{=} \left( (1,1) \circ \psi^{-1}_{ \psi_{(1,x^{-1})}(b^{-1},1)}(a^{-1},1) \right)^{-1}\circ \delta_x(b)\\
    &\overset{\eqref{linking3}}{=} \left( (1,1) \circ \psi^{-1}_{(b^{-1},1)\cdot \theta^{-1}_{(b^{-1},1)}(1,x)}(a^{-1},1)\right)^{-1} \circ \delta_x(b) .
\end{align*}
As $\psi$ is a left action of $(S,\cdot)$ on the set $S$ and $\psi_{(a,1)}=\id_S$ for all $a \in A$, one obtains that $$ \psi^{-1}_{(b^{-1},1)\cdot\theta^{-1}_{(b^{-1},1)}(1,x)} = \psi^{-1}_{(1,1)\cdot \theta_{(b,1)}(1,x)}, $$ where we also applied that $\theta$ is a left action of $(S,\circ)$ on $S$ to cancel the inverses of $\theta$ and $b$. Hence, reprising the previous calculation,
\begin{align*}
    \delta_x(a\circ b) &=\left( (1,1) \circ \psi^{-1}_{(b^{-1},1)\cdot \theta^{-1}_{(b^{-1},1)}(1,x)}(a^{-1},1)\right)^{-1} \circ \delta_x(b)\\
    &=\left( (1,1) \circ \psi^{-1}_{(1,1)\cdot \theta_{(b,1)}(1,x)}(a^{-1},1)\right)^{-1} \circ \delta_x(b)\\
    &= \left( (1,1) \circ \psi^{-1}_{(1,\sigma_b(x))}(a^{-1},1)\right)^{-1} \circ \delta_x(b),
\end{align*}
where one applies the definition of $\sigma$. As $\psi$ is an action of $G$ on $S$, one finds from the previous that \begin{align*}
    \delta_x(a\circ b)  &= \left( (1,1) \circ \psi_{(1,\sigma_b(x)^{-1})}(a^{-1},1)\right)^{-1} \circ \delta_x(b)\\
    &= \delta_{\sigma_b(x)}(a) \circ \delta_x(b).
    \end{align*}
    Hence, the quadruple $(G,A,\sigma,\delta)$ is a matched pair of groups, as claimed. 

    Let $(S,s')$ be the solution associated to the matched pair of groups $(G,A,\sigma,\delta)$ constructed in Theorem~\ref{thm:matched} with corresponding map $\theta'$ and semigroup $(S,\cdot')$. 
    It is clear that the semigroups $(S,\cdot)$ and $(S,\cdot')$ coincide. Hence, it remains to show that for any $(a,x)=(a,1)(1,x) \in S$, one has that $\theta'_{(a,x)}=\theta_{(a,x)}$.
    Let $b \in A$ and $y \in G$. Then by relation (\ref{PE2}), it is sufficient to show that $\theta'_{(a,x)}(b,1) = \theta_{(a,x)}(b,1)$ and $(1,1)\cdot \theta'_{(a,x)}(1,y)=(1,1) \cdot \theta_{(a,x)}(1,y)$. We start with the former. By definition of $\theta'$, we have that \begin{align*}
        \theta'_{(a,x)}(b,1) &= (b \circ (\delta_x(a))^{-1},1)\\
        &= b \circ (1,1) \circ \psi_{(1,x^{-1})}(a^{-1},1) .
    \end{align*}
    As $(1,x)$ and $(1,1)$ are left identities for $(S,\circ)$ one can rewrite the latter into \begin{align*}
        \theta'_{(a,x)}(b,1) &=(b,1) \circ (1,x) \circ \psi_{(1,x^{-1})}(a^{-1},1) \\
        &\overset{\eqref{linking1}}{=} (b,1) \circ \theta_{(a,1)}(1,x)\\
        &= \theta^{-1}_{\theta_{(a,1)}(1,x)}(b,1)\\
        &\overset{\eqref{linking2}}{=} \theta_{(a,x)}(b,1),
    \end{align*}
    which shows the first required equality.  Finally, let $y \in G$. Then, \begin{align*}
        (1,1)\cdot\theta'_{(a,x)}(1,y) &= (1, \sigma_{(\delta_x(a),1)}(y))\\
        &= (1,1)\cdot \theta^{-1}_{(1,1) \circ \psi_{(1,x^{-1})}(a^{-1},1)}(1,y),
        \end{align*}
        where the latter equality holds by equality~(\ref{eq:recalibrate}). Continuing the chain of equalities,
        \begin{align*}
        (1,1)\cdot\theta'_{(a,x)}(1,y) &= (1,1)\cdot \theta^{-1}_{(1,1) \circ \psi_{(1,x^{-1})}(a^{-1},1)}(1,y)\\        
        &\overset{\eqref{linking1}}{=} (1,1)\cdot \theta^{-1}_{(1,x) \circ \psi_{(1,x^{-1})}(a^{-1},1)}(1,y)\\
        &\overset{\eqref{linking2}}{=}(1,1)\cdot \theta^{-1}_{\theta_{(a,1)}(1,x)}(1,y)\\
        &=(1,1) \cdot \theta_{(a,x)}(1,y),
    \end{align*}
    which shows that $\theta'=\theta$. Hence, the result follows.
\end{proof}

Let us now discuss isomorphisms of irretractable solutions. We will see that
isomorphic irretractable solutions correspond to isomorphic matched pairs of
groups.

Here (and throughout), the notion of isomorphism for matched pairs is the natural one:
we say that $(A,G,\sigma,\delta)$ and $(A',G',\sigma',\delta')$ are \emph{isomorphic}
matched pairs if there exist group isomorphisms $f:A\to A'$ and $g:G\to G'$ such that
\[
g(\sigma_a(x))=\sigma'_{f(a)}(g(x))
\quad\text{and}\quad
f(\delta_x(a))=\delta'_{g(x)}(f(a))
\]
for all $a\in A$ and $x\in G$.

\begin{pro}
    Let $(A,G,\sigma,\delta)$ and $(A',G',\sigma', \delta')$ be isomorphic matched pairs of groups. Then $(A\times G, s)$ and $(A'\times G',s')$, the irretractable bijective solutions to the PE as defined in Theorem~\ref{thm:matched}, are isomorphic.
\end{pro}
\begin{proof}
    Let $a,b\in A$ and $x,y\in G$. By definition of isomorphism of matched pairs we have 
    \begin{align*}
        f(a,x)f(b,y) &= (f_1(a),f_2(x))(f_1(b),f_2(y)) = (f_1(a), f_2(x)f_2(y))\\ &
        = (f_1(a), f_2(xy))
        = f(a,xy) = f((a,x)(b,y)).
    \end{align*}
    Moreover, 
    \begin{align*}
        f\theta_{(a,x)}(b,y) &= f(b(\delta_x(a))^{-1},\sigma_{\delta_x(a)}(y))
        = (f_1(b(\delta_x(a))^{-1}), f_2(\sigma_{\delta_x(a)}(y)))\\
        &= (f_1(b)f_1(\delta_x(a))^{-1}, \sigma'_{f_1\delta_x(a)}f_2(y))\\
        &= (f_1(b)(\delta'_{f_2(x)}f_1(a))^{-1}, \sigma'_{\delta'_{f_2(x)}f_1(a)}f_2(y))\\
        &= \theta'_{(f_1(a),f_2(x))}(f_1(b),f_2(y)) = \theta'_{f(a,x)}f(b,y).
    \end{align*}
    Hence,
    \begin{align*}
        (f\times f)s((a,x),(b,y)) 
        &= (f\times f)((a,x)(b,y),\theta_{(a,x)}(b,y))\\
        &= (f(a,x)f(b,y), \theta'_{f(a,x)}f(b,y))
        = s'(f(a,x),f(b,y)) \\
        &= s'(f\times f)((a,x),(b,y)).
    \end{align*}
    Hence $(S,s)$ and $(S',s')$ are isomorphic solutions.
\end{proof}

Next, we will prove the converse. First, note that if $(S,s)$ and $(S',s')$ are isomorphic solutions and $(S,s)$ is bijective and irretractable then so is $(S',s')$.
\begin{pro}\label{pro:isoirr}
    Let $(S,s)$ and $(S',s')$ be finite irretractable solutions to the PE. Then there exists an isomorphism of the underlying matched pairs.
\end{pro}
\begin{proof} Without loss of generality (see Theorem~\ref{thm:sol2mathc}) let us assume that $(S,s)$ and $(S,s')$ are, respectively, the solutions associated to the matched pairs $(A,G,\sigma,\delta)$ and $(A',G',\sigma',\delta')$.
    Let $f$ be the isomorphism between the solutions $(S,s)$ and $(S',s)$. Clearly, $f$ defines an isomorphism of the left groups $(S,\cdot)$ and $(S',\cdot)$. Since $f$ maps $A\times \left\lbrace 1 \right\rbrace = E(S,\cdot)$ onto $E(S',\cdot)=A' \times \left \lbrace 1 \right\rbrace$, $f$ induces a bijective map from $A$ to $A'$ which we denote $f_1$. Note that as $(1,1)$ is the unique idempotent in $S$ with $\theta_{(1,1)}=\id_S$, one finds that $f(1,1)=(1,1)$. Hence, $f$ restricts to an isomorphism between $(1,1)S=\left\lbrace 1 \right\rbrace \times  G$ and $\left\lbrace 1 \right\rbrace \times G'$. Thus, $f$ induces an isomorphism $f_2$ from $G$ to $G'$. To summarize, for $a\in A$ and $x \in G$ one has $f(a,x)= (f_1(a),f_2(x))$. Next we show that $f_1$ is an isomorphism between $(A,\circ)$ and $(A',\circ)$. Let $a,b \in A$. Then, $$ (f \times f)s((a^{-1},1),(b,1)) = (f\times f) ((a,1),(b\circ a,1)) = ((f_1(a),1),(f_1(b\circ a),1)).$$ On the other hand, we obtain that $$ s' \circ (f \times f)((a^{-1},1),(b,1)) = ((f_1(a),1),(f_1(b)\circ f_1(a),1)).$$ Hence, $f_1$ is an isomorphism between $A$ and $A'$. It remains to show that $f$ respects the actions $\sigma$ and $\delta$. Let $a \in A$ and $x\in G$. Then $$ (f\times f) s((a,x),(1,1)) = (f\times f)( (a,x), ( \delta_x(a)^{-1},1))= (f(a,x),(f_1(\delta_x(a)^{-1}),1)).$$ Similarly, $$ s' (f\times f) ((a,x),(1,1)) = ((f_1(a),f_2(x)),(\delta_{f_2(x)}'(f_1(a))^{-1},1)).$$ As $f_1$ is a group morphism, this shows that $f_1 \delta_x(a) = \delta_{f_2(x)}'(f_1(a))$. Applying a similar reasoning on the tuple $((a,1),(1,x))$ one finds that $f_2 \sigma_a(x)= \sigma_{f_1(a)}'(f_2(x))$. Hence, $f$ is an isomorphism of the matched pairs $(A,G,\sigma,\delta)$ and $(A',G',\sigma',\delta')$. 
\end{proof}

\section{A complete description of finite bijective solutions}

In this section, we first describe extensions of a finite irretractable bijective solution and then we provide a full classification of all finite bijective solutions (up to isomorphism).

\begin{lem}\label{lem:extviakern}
    Let $(S,s)$ be a finite bijective solution to the PE. Let $e \in E(S,\cdot)$. Then, the equivalence class of $e$ under the retract $\approx$ has exactly $|\ker \theta \cap E(S,\cdot) |$ elements. In particular, denoting by $\overline{E}$ the set of idempotents of the retract $\overline{S}$, one can identify $E(S,\cdot)$ with the Cartesian product of $\ker \theta \cap E(S,\cdot)$ and $\overline{E}$.
\end{lem}
\begin{proof}
    Let $a \in E(S,\cdot)$. Let $b \in S$ be in the equivalence class of $a$ under the retract $\approx$, i.e. $\theta_a=\theta_b$. Then, by definition of $\approx$, one has that $b \in E(S,\cdot)$. Recall that $(S,\circ)$ is a left group. Hence, there exists an $x\in S$ such that $x \circ a = b$. As $\theta$ is an morphism of $(S,\circ)$ to $\Sym(S)$, we find that  $$ \theta_a=\theta_b = \theta_{x \circ a}=\theta_x \theta_a.$$ Thus, $\theta_x \in \ker \theta$. Moreover, by Proposition~\ref{pro:idempotentsubsemi}, $E(S,\cdot)$ is a sub left group of $(S,\circ)$, which shows that $x \in E(S,\cdot)\cap \ker \theta$. This shows that there are at most $|\ker \theta \cap E(S,\cdot)|$ elements in the equivalence class of $a$. Vice versa, one sees easily that for any $x \in \ker \theta \cap E(S,\cdot)$ one has that $\theta_{x \circ a}=\theta_a$ and $x \circ a \in E(S,\cdot)$ for any $a \in E(S,\cdot)$. Hence, $x \circ a$ is in the equivalence class of $e$ under $\approx$. As $(S,\circ)$ is a left group, one obtains at least $|\ker \theta \cap E(S,\cdot)|$ elements in the equivalence class of $a$, which proves the result.
\end{proof}

Recall that, for a finite irretractable bijective solution $(S,s)$, the set $S$ can be identified with the Cartesian product of $A$ and $G$ as in Lemma~\ref{lem:AandG}. 

\begin{pro}\label{pro:extension}
    Let $(S,s)$ be a finite irretractable bijective solution to the PE, let $X$ be a finite set, and let $\varphi_a$ be a permutation of $X$, for every $a \in A$. Then $s': (X\times S)\times (X\times S) \to (X\times S)\times (X\times S)$ defined by
    \begin{align*}
        s'((\alpha, a, x), (\beta, b, y)) = ((\alpha, a , xy), (\varphi^{-1}_{b(\delta_x(a))^{-1}}\varphi_b(\beta), \theta_{(a,x)}(b,y)))
    \end{align*}
    is a solution to the PE.  
\end{pro}
\begin{proof}
    By Theorem~\ref{thm:sol2mathc}, $s'$ can be written as follows
    \begin{align*}
         s'((\alpha, a, x), (\beta, b, y)) = ((\alpha, a , xy), (\varphi^{-1}_{b(\delta_x(a))^{-1}}\varphi_b(\beta), b(\delta_x(a))^{-1}, \sigma_{\delta_x(a)}(y))).
    \end{align*}
   Clearly, $X\times S$ has a structure of a left group with respect to the multiplication given by $(\alpha, a, x)(\beta, b, y)=(\alpha, a , xy)$ and equality~\eqref{PE1} is satisfied. 
    Hence, it remains to show equalities~\eqref{PE2} and \eqref{PE3}. We start with equality~\eqref{PE2}. Let $\alpha, \beta, \gamma \in X, a,b,c \in A$ and $x,y,z \in G$. Then, \begin{align*}
        \theta'_{(\alpha,a,x)}((\beta,b,y)(\gamma,c,z)) &= \theta'_{(\alpha,a,x)}(\beta,b,yz)\\ &=(\varphi_{b\delta_x(a)^{-1}}^{-1}\varphi_b(\beta),\theta_{(a,x)}(b,yz))\\
        &=(\varphi_{b\delta_x(a)^{-1}}^{-1}\varphi_b(\beta),\theta_{(a,x)}(b,y)\theta_{(a,xy)}(c,z)).
    \end{align*}
    On the other hand, one has that \begin{align*} &\theta'_{(\alpha,a,x)}(\beta,b,y) \theta'_{(\alpha,a,xy)}(\gamma,c,z)\\ =& (\varphi_{b\delta_x(a)^{-1}}^{-1} \varphi_b(\beta), \theta_{(a,x)}(b,y)) (\varphi_{c\delta_{xy}(a)^{-1}}^{-1}\varphi_c(\gamma),\theta_{(a,xy)}(c,z)) \\ =& (\varphi_{b\delta_x(a)^{-1}}^{-1}\varphi_b(\beta),\theta_{(a,x)}(b,y)\theta_{(a,xy)}(c,z)).
    \end{align*}
    As equality~\eqref{PE2} holds for $(S,s)$, it follows indeed that equality~\eqref{PE2} holds for $(X \times S,s')$. 
    It remains to prove equality~\eqref{PE3}. Note that, by the definition of $\theta'$, $$ \theta'_{\theta'_{(\alpha,a,x)}(\beta,b,y)}\theta'_{(\alpha,a,xy)}(\gamma,c,z) = \theta'_{(\zeta, \theta_{(a,x)}(b,y))}(\varphi_{c \delta_{xy}(a)^{-1}}^{-1}\varphi_c(\gamma), \theta_{(a,xy)}(c,z)),$$ 
    for some $\zeta \in X$. This in turn is equal to $$ (\varphi_{c\delta_{xy}(a)^{-1}\delta_{\sigma_{\delta_x(a)}(y)}(b\delta_x(a)^{-1})^{-1}}^{-1}\varphi_{c\delta_{xy}(a)^{-1}}\varphi_{c\delta_{xy}(a)^{-1}}^{-1}\varphi_c(\gamma),\theta_{\theta_{(a,x)}(b,y)}\theta_{(a,xy)}(c,z)).$$
    This can be simplified to $$ (\varphi_{c\delta_{xy}(a)^{-1}\delta_{\sigma_{\delta_x(a)}(y)}(b\delta_x(a)^{-1})^{-1}}^{-1}\varphi_c(\gamma),\theta_{\theta_{(a,x)}(b,y)}\theta_{(a,xy)}(c,z)).$$ Note that by equality~\eqref{eq:deltaprod}, one has that $$ \delta_y(b) = \delta_{\sigma_{\delta_x(a)}(y)}(b\delta_x(a)^{-1}) \delta_y\delta_x(a) = \delta_{\sigma_{\delta_x(a)}(y)}(b\delta_x(a)^{-1})\delta_{xy}(a). $$
    Applying this formula on the expression, one simplifies it to $$ \varphi_{c \delta_y(b)^{-1}}^{-1}\varphi_c(\gamma), \theta_{\theta_{(a,x)}(b,y)}\theta_{(a,xy)}(c,z)). $$ As $$\theta'_{(\beta,b,y)}(\gamma,c,z) = ( \varphi_{c\delta_y(b)^{-1}}^{-1}\varphi_c(\gamma),\theta_{(b,y)}(c,z)),$$ one finds that equality~\eqref{PE3} holds for $(X\times S,s')$ as it holds for $(S,s)$. Hence, $(X\times S,s')$ is a solution to the PE. Moreover, $(X\times S,s')$ is bijective. Indeed, let $(\alpha,a,x),(\beta,b,y)\in X \times S$ be given. Then we need to show there exist unique $(\gamma,c,z), (\zeta,d,w)\in X \times S$ such that $$s'((\gamma,c,z),(\zeta,d,w))=((\alpha,a,x),(\beta,b,y)).$$ One finds that this implies that $$ s((c,z),(d,w))=((a,x),(b,y)).$$ As $(S,s)$ is bijective, it is known that there exist such unique pairs $(c,z),(d,w) \in S$. By the definition of $s'$ it is clear that $\gamma=\alpha$. Finally, as the maps $\varphi$ are permutations and depend only on the pairs $(c,z),(d,w)$, it follows that there indeed exists a unique $\zeta \in X$ such that $$ s((\gamma,c,z),(\zeta,d,w)) = ((\alpha,a,x),(\beta,b,y)),$$ proving the result.
\end{proof}

The solution $(S',s')$ constructed in Proposition~\ref{pro:extension} will be called the \emph{extension of $(S, s)$} by $X$ and $\varphi=\{\varphi_a\colon a \in A\}$ and denoted by
$\operatorname{Ext}_X^{\varphi}(S,s)$.

\begin{pro}\label{pro:isoextensions}
    Let $(S,s)$ be a finite irretractable bijective solution to the PE corresponding to the matched pair $(A,G,\sigma,\delta)$. Let $X$ be a non-empty set. 
    Then, $\operatorname{Ext}_X^{\varphi}(S,s)$ is isomorphic to $\operatorname{Ext}_X^{\rho}(S,s)$, for any sets of bijections $\varphi$ and $\rho$.
\end{pro}
\begin{proof}
    Define the map $\xi: \operatorname{Ext}_X^{\varphi}(S,s) \rightarrow \operatorname{Ext}_X^{\rho}(S,s)$ by $$\xi(\alpha,x,a)=(\rho_x^{-1}\varphi_x(\alpha),x,a).$$ Clearly, $\xi$ is a bijective  map. 
    Denote $\operatorname{Ext}_X^{\varphi}(S,s)=(X\times S,s_\varphi)$ and $\operatorname{Ext}_X^{\rho}(S,s)=(X\times S, s_{\rho})$. Then for any $(\alpha,a,x),(\beta,b,y) \in X \times S$, by the definition of $s_{\varphi},s_{\rho}$ and $\xi$, one has that
    \begin{align*}
        &(\xi \times \xi)s_{\varphi}((\alpha,a,x),(\beta,b,y))\\ 
        &= (\xi \times \xi) ((\alpha,a,xy),(\varphi_{b \delta_{x}(a)^{-1}}^{-1}\varphi_b(\beta),b(\delta_x(a))^{-1}, \sigma_{\delta_x(a)}(y)))\\
        &= ( (\rho_a^{-1}\varphi_a(\alpha),a,xy), (\rho_{b\delta_x(a)^{-1}}^{-1}\varphi_{b\delta_x(a)^{-1}}\varphi_{b\delta_x(a)^{-1}}^{-1}\varphi_b(\beta),b(\delta_x(a))^{-1}, \sigma_{\delta_x(a)}(y)))\\
        &= ( (\rho_a^{-1}\varphi_a(\alpha),a,xy), (\rho_{b\delta_x(a)^{-1}}^{-1}\varphi_b(\beta),b(\delta_x(a))^{-1}, \sigma_{\delta_x(a)}(y)))\\
        &= ( (\rho_a^{-1}\varphi_a(\alpha),a,xy), (\rho_{b\delta_x(a)^{-1}}^{-1}\rho_b \rho_b^{-1} \varphi_b(\beta),b(\delta_x(a))^{-1}, \sigma_{\delta_x(a)}(y)))\\
        &= s_{\rho}((\rho_a^{-1}\varphi_a(\alpha),a,x), (\rho_b^{-1} \varphi_b(\beta), b, y))\\
        &= s_{\rho}(\xi \times \xi) ((\alpha,a,x), (\beta,b,y))
    \end{align*}
    This shows that $\xi$ is an isomorphism of solutions.
\end{proof}

Hence, up to isomorphism there is only one extension for a given finite irretractable bijective solution $(S,s)$ and a non-empty set $X$.

Now we show that every finite bijective solution to the PE can be obtained as the extension of its retract. In view of Proposition~\ref{pro:isoextensions}, this is in a big contrast to the case of the retract operation on set-theoretic solutions to the Yang--Baxter Equation, see the discussion at the beginning of Section~\ref{sect3}. 

\begin{thm}  \label{thm:summary}
    Let $(S,s)$ be a finite bijective solution to the PE and let $(\overline{S},\overline{s})$ be its retract. Then $(S,s)$ is an extension of $(\overline{S}, \overline{s})$ by $X=\ker \theta \cap E(S,\cdot)$.
\end{thm}
\begin{proof}
    Note that, by Lemma~\ref{lem:extviakern}, one can identify the set $S$ with $X \times A \times G$, where the decomposition $\overline{S}=A\times G$ with matched pair of groups $(A,G,\sigma,\delta)$ is such as in Theorem~\ref{thm:sol2mathc} and $X$ is the set $\ker\theta \cap E(S,\cdot)$. Let $\alpha, \beta \in X, a,b \in A$ and $x,y \in G$. Then there exist $\gamma \in X, c \in A$ and $z \in G$ such that $$ \theta_{(\alpha, a, x)}(\beta,b,y)=(\gamma,c,z).$$ As taking the retract of a solution is a morphism of solutions, one finds that $(c,z) = \theta_{(a,x)}(b,y)$. Hence, $$ \theta_{(\alpha,a,x)}(\beta,b,y)=(\gamma,\theta_{(a,x)}(b,y)).$$ As $\theta_{(\alpha,a,x)} = \theta_{(\alpha',a,x)}$ for all $\alpha' \in X$, the element $\gamma$ is independent of $\alpha$. Moreover, by equality~\eqref{PE2}, one has that $$ \theta_{(\alpha,a,x)}(\beta,b,y) = \theta_{(\alpha,a,x)}(\beta,b,1)\theta_{(\alpha,a,x)}(\beta,b,y),$$ which shows that $\gamma$ is independent of $y$. Finally, notice that $$\theta_{(\alpha,a,x)}=\theta_{\theta_{(\alpha, a,x)}(\alpha,1,1)^{-1}}= \theta_{(\alpha',\delta_{x}(a),1)}.$$ Thus, $\gamma$ depends on $\delta_x(a)$ instead of on $a$ and $x$ separately. Hence, denote $$ \theta_{(\alpha,a,x)}(\beta,b,y)=(\eta_{\delta_x(a),b}(\beta),\theta_{(a,x)}(b,y)).$$ Note that the maps $\eta$ are bijective, as $s$ is bijective. We examine equality~\eqref{PE3} to obtain conditions on the maps $\eta$. Let $\alpha,\beta,\gamma \in X, a,b,c \in A$ and $x,y,z \in G$. Then for some $\zeta \in X$ one has \begin{align*}
    &\theta_{\theta_{(\alpha,a,x)}(\beta,b,y)} \theta_{(\alpha,a,xy)}(\gamma,c,z)\\ =& \theta_{(\zeta,b\delta_x(a)^{-1},\sigma_{\delta_x(a)}(y))}(\eta_{\delta_{xy}(a),c}(\gamma),c\delta_{xy}(a)^{-1},\sigma_{\delta_{xy}(a)}(z))\\
    =& (\eta_{\delta_{\sigma_{\delta_x(a)}(y)}(b\delta_x(a)^{-1}),c\delta_{xy}(a)^{-1}}\eta_{c\delta_{xy}(a)^{-1}}(\gamma), d,w)\\
    =&(\eta_{\delta_y(b)\delta_{xy}(a)^{-1}, c\delta_{xy}(a)^{-1}}\eta_{\delta_{xy}(a),c}(\gamma), d,w),
    \end{align*}
    with $d \in A$ and $w\in G$. Calculating the right hand side of equality~\eqref{PE3} and equating its first component with the first component in the previous calculation, one finds the equality $$ \eta_{\delta_y(b)\delta_{xy}(a)^{-1}, c\delta_{xy}(a)^{-1}}\eta_{\delta_{xy}(a)^{-1},c} = \eta_{\delta_y(b),c}.$$ Note that as $\delta_{xy}$ and $\delta_y$ are bijective maps, this can be rewritten such that for any $a,b,c \in A$ one has that  $$  \eta_{ba^{-1},ca^{-1}}\eta_{a^{-1},c} = \eta_{b,c}. $$ Note that $\eta_{1,c}(\gamma) = \gamma$ as this corresponds to $ \theta_{(\alpha,1,1)}(\gamma,c,1)=(\gamma,c,1)$. Moreover, if we let $a=c$, then we see that 
    $$\eta_{bc^{-1},1}\eta_{c,c}=\eta_{b,c}.$$ Applying this equality for $b=1$, we find that $$\eta_{c^{-1},1}\eta_{c,c}=\eta_{1,c}=\id_X.$$ 
    Hence, $$\eta_{b,c}=\eta_{bc^{-1},1}\eta_{c,c}=\eta_{bc^{-1},1}\eta_{c^{-1},1}^{-1}.$$ 
    Denote $\varphi_t=\eta_{t^{-1},1}^{-1}$. Then, one reformulates the equality to $$ \eta_{b,c}= \varphi_{cb^{-1}}^{-1} \varphi_{c}.$$ Applying this in the formula for $\theta$, we obtain that $$\theta_{(\alpha,a,x)}(\beta,b,y) = ((\varphi_{b\delta_x(a)^{-1}}^{-1}\varphi_b(\beta),\theta_{(a,x)}(b,y)).$$ Thus, we have shown that every finite bijective solution can be written as an extension of its retract as in Proposition~\ref{pro:extension}.
\end{proof}

In summary, one obtains our main result which is the following complete characterization of finite bijective solutions to the PE.  
\begin{thm}\label{thm:iso}
    Let $(A,G,\sigma,\delta)$ be a matched pair of groups. Then, for any non-empty set $X$ and set of permutations $\left\lbrace \varphi_a \mid a \in A\right\rbrace$, the following defines a bijective solution to the PE on $X\times A \times G$, $$ s((\alpha,a,x),(\beta,b,y))= ((\alpha, a , xy), (\varphi^{-1}_{b(\delta_x(a))^{-1}}\varphi_b(\beta), b(\delta_x(a))^{-1}, \sigma_{\delta_x(a)}(y))).$$ Vice versa, every finite bijective solution $(S,s)$ to the PE is such a solution, where the corresponding matched pair $(A,G,\sigma,\delta)$ is determined by the retract of $(S,s)$, as in Theorem~\ref{thm:sol2mathc}.
\end{thm}

We conclude this section with an explicit description of isomorphism classes of finite bijective solutions to the PE. 

\begin{cor}\label{cor:classif-iso}
    Let $(S,s)$ and $(S',s')$ be finite bijective solutions to the PE.
    Let $(A,G,\sigma,\delta)$ and $(A',G',\sigma',\delta')$ be the matched pairs determined
    by the retracts of $(S,s)$ and $(S',s')$, respectively, and let $X,X'$ be the corresponding extension sets in Theorem~\ref{thm:iso}.
    Then the following are equivalent:
    \begin{enumerate}
        \item $(S,s)\cong (S',s')$ as set-theoretic solutions to the PE,
        \item $|X|=|X'|$ and $(A,G,\sigma,\delta)\cong (A',G',\sigma',\delta')$ as matched pairs.
    \end{enumerate}
    In particular, the isomorphism classes of finite bijective solutions to the PE are in a bijection with pairs consisting of an isomorphism class of matched
    pairs of groups and a positive integer.
\end{cor}
\begin{proof}
    By Theorem~\ref{thm:iso}, any finite bijective solution is isomorphic to one of the form $X\times A\times G$ with $X$ a set and $(A,G)$ a matched pair of groups determined by its retract. Its isomorphism class is determined uniquely by the isomorphism class of the matched product, Proposition~\ref{pro:isoirr}, and the cardinality of $X$, Proposition~\ref{pro:isoextensions}.
\end{proof}

\section{Concluding remarks and examples}   \label{comments}

We first show that, as a direct consequence of Proposition~\ref{pro:isoextensions} and Theorem~\ref{thm:summary}, one recovers the classification of involutive solutions obtained in \cite{MR4170296}, in the finite case.

Let $(S,s)$ be a finite bijective solution to the PE such that $s^{2}=\id$. Then, by Theorem~\ref{thm:summary}, $(S,s)$ is an extension of its retract $(\overline{S},\overline{s})$. Since $s^{2}=\id$ then also $\overline{s}^{2}=\id$. 
By Theorem~\ref{thm:sol2mathc} we have
$$\overline{s}^2 ((a,x),(b,y)) = ((a,xy\sigma_{\delta_{x}(a)}(y)),
(b(\delta_{x}(a))^{-1} (\delta_{xy}(a))^{-1},
\sigma _{\delta_{xy}(a)}\sigma_{\delta_{x}(a)}(y))).$$ 
Hence, $y\sigma_{\delta_{x}(a)}(y)=1$. Since each $\delta_x$ is bijective,  it follows that $\sigma_{c}(y) = y^{-1}$, for every $y\in G$ and $c\in A$. Moreover, we obtain $(\delta_{x}(a))^{-1} (\delta_{xy}(a))^{-1}=1$, for all $x,y\in G$ and $a\in A$. Therefore $\delta_{x} = \delta_z$ for all $x,z\in G$ and $ (\delta_{x}(a))^{2} =1$. Consequently, $A$ is an elementary abelian $2$-group. 
Since $\sigma$ is a left action of $A$ on $G$, we have $\sigma_{ab} = \sigma_a \sigma_b$ and hence $y=y^{-1}$ for every $y\in G$ and $\sigma_a =\id $ for every $a\in A$. 
In particular, $G$ is an elementary abelian $2$-group. 
Since $\delta$ is a right action of $G$ on $A$, we get that $\delta_{xy} = \delta_y \delta_x$, 
so that $\delta_x =\id $ for every $x\in G$. 
Therefore  
$G\times A$ is an ordinary direct product of groups.
Conversely, it is easy to see that if $G$ and $A$ are elementary abelian $2$-groups, $\sigma_a=\id$ for every $a\in A$, and $\delta_x=\id$  for every $x\in G$, then conditions  \eqref{eq:sigmaprod} and \eqref{eq:deltaprod} are satisfied and the solution defined in Theorem~\ref{thm:matched} is involutive. So, in view of Proposition~\ref{pro:isoextensions} and Theorem~\ref{thm:summary}, we recover the main result of \cite{MR4170296}.

We next describe several sources of matched products of groups, and hence of explicit irretractable
bijective solutions to the PE via Theorem~\ref{thm:matched}. We start with a basic construction (the semidirect product) and
then discuss examples arising from skew braces via holomorph/permutation-group methods.

\begin{exa}
    Let $G$ and $A$ be groups and assume $A$ acts on $G$ by automorphisms, equivalently we are given a group homomorphism $\sigma:A \to \operatorname{Aut}(G)$. Define a right action $\delta \colon G \to \Sym(A)$ to be trivial, i.e.\ $\delta_x=\id_A$ for all $x\in G$. Then $(G,A,\sigma,\delta)$ is a matched pair of groups: condition~\eqref{eq:sigmaprod} reduces to $\sigma_a(xy)=\sigma_a(x)\sigma_a(y)$, i.e. $\sigma_a \in \operatorname{Aut}(G)$ and condition~\eqref{eq:deltaprod} is trivially satisfied. The corresponding irretractable solution according to Theorem~\ref{thm:iso} is $s:(A\times G)^2\to (A\times G)^2$ defined by
    \begin{align*}
        s((a,x),(b,y)) = ((a,xy),(ba^{-1}, \sigma_{a}(y))).
    \end{align*}
    In this case, our matched product of groups can be identified with the inner semidirect product $G\rtimes A$ of its normal subgroup $G$ and a subgroup $A$ and $\sigma_a (x) = axa^{-1}$ for $a\in A, x\in G$.
\end{exa}

More generally, exact factorizations provide one of the most concrete sources of matched pairs.
Recall that a group $S$ \emph{factorizes} through subgroups $G$ and $A$ if $S=GA$,
and the factorization is \emph{exact} if moreover $G\cap A=\{1\}$ (equivalently:
every $s\in S$ has a unique decomposition $s=xa$ with $x\in G$, $a\in A$).
In this situation one defines mutual actions by the unique decomposition
\begin{align*}
    a x = (a \triangleright x)\,(a \triangleleft x),
\qquad a\in A,\ x\in G,
\end{align*}
with $a\triangleright x\in G$ and $a\triangleleft x\in A$. Then $(G, A, \triangleleft, \triangleright)$ defines a matched pair of groups and the resulting matched product
$G \bowtie A$ recovers $S$.

Clearly, there are many matched pairs of groups that do not come from (semi)direct
products, in the sense that neither factor is normal.  For example, the alternating group $A_5$ admits an exact factorization $A_5 = AG$ with $G\cong C_5$ and $A\cong A_4$. In this case, both actions $\sigma$ and $\delta$ are nontrivial.

A particularly rich and well-studied source of such structures comes from skew braces, which are
central in the theory of set-theoretic solutions to the Yang--Baxter Equation: given a skew brace one
obtains a (non-degenerate) set-theoretic solution, and conversely every non-degenerate such solution
yields a skew brace \cite{GuVe2017} (see also the earlier group-with-braiding approach of
Lu--Yan--Zhu~\cite{Lu2000}).
Braces, and later skew braces, were introduced by Rump~\cite{zbMATH05118810}
and by Guarnieri--Vendramin~\cite{GuVe2017}, and they have played a central role in the subject; see,
for instance, \cite{zbMATH07408036,CJO2014, GI2018, zbMATH07075038} and the
recent survey~\cite{zbMATH07919707} and references therein.
Moreover, a skew brace canonically determines a matched pair of groups in which the multiplicative
group acts on itself from the left and from the right \cite[Lemma~5.9]{MR3763907}. By
Theorem~\ref{thm:iso}, this matched pair also produces an irretractable set-theoretic solution to the PE.
This highlights an interesting connection between
set-theoretic solutions to the Yang--Baxter Equation and to the PE; some connections between the two
equations in their general form have already been pointed out in \cite{MR3957155,MR1273649}. 

\begin{exa}\label{ex:skewbraces}
    Let $(B,+,\circ)$ be a skew brace \cite{GuVe2017}. 
    Recall that we can define a left action $\lambda: (B,\circ) \to \Aut(B,+)$ by $\lambda_a(b)=-a+a\circ b$ and a right action $\rho: (B,\circ)\to \Sym(B)$ by $\rho_b(a)=(\lambda_a(b))'\circ a \circ b$.
    It is known that $(B,\circ)$ with respect to the actions 
    $\lambda$ and $\rho$ forms a matched pair of groups, see for instance \cite[Lemma 5.9]{MR3763907}.
    Therefore, any skew brace yields an irretractable bijective solution $(B\times B, s)$ to the PE such that
    \begin{align*}
        s((a,x),(b,y)) &= ((a, x\circ y),(b\circ (\rho_x(a))', \lambda_{\rho_x(a)}(y)))\\
        &=((a, x\circ y),(b\circ (x'\circ a- x'),(a'+x)'\circ(a'+y))).
    \end{align*}
\end{exa}

Skew braces can also be described in terms of regular subgroups of holomorphs. Fix an additive group
$(B,+)$. Then skew brace structures on $B$ correspond to regular subgroups of the holomorph
$\Hol(B,+)=(B,+)\rtimes \Aut(B,+)$ up to conjugacy; see, for instance,
\cite{GuVe2017,zbMATH07224504}. Each such regular subgroup determines a skew brace, and
Example~\ref{ex:skewbraces} then yields an explicit irretractable solution to the PE.

This holomorph perspective extends further to the notion of skew bracoids. Let $H$ be a group and let
$G\le \Hol(H)$ act transitively on the underlying set of $H$ via the natural holomorph action. Following
\cite{221207361}, the pair $(G,H)$ is called a \emph{skew bracoid} if this action is transitive; it is
said to \emph{contain a brace} if $G$ contains a regular subgroup of $\Hol(H)$. If $(G,H)$ is a skew
bracoid containing a brace, then $G$ admits an exact factorization $G=HS$, where $S=\Stab_G(1)$, which
yields a matched pair between $H$ and $S$ \cite[Proposition~2.2]{221207361}. Consequently, skew bracoids
containing a brace provide further examples of matched pairs and, via Theorem~\ref{thm:iso}, of
irretractable bijective solutions to the PE.

This also creates a bridge to Hopf--Galois theory. By the Greither--Pareigis theorem, Hopf--Galois
structures on a finite separable field extension correspond to transitive subgroups of a suitable
permutation group. Byott's reformulation \cite{zbMATH01014452} expresses this in terms of transitive
subgroups inside suitable holomorphs, reducing classification to the study of transitive holomorph
subgroups. In this context, algorithms that enumerate transitive subgroups have been developed; see for
instance \cite{Da25x,zbMATH06607874}. These computational outputs therefore also provide effective input
for constructing matched pairs and hence irretractable solutions to the PE via Theorem~\ref{thm:iso}.

We conclude by exploring the relation between set-theoretic solutions to the PE and Hopf algebras with
positive bases. Recently, the first author and Janssens \cite{CJ26x} investigated the relationship
between bijective set-theoretic solutions to the PE and Hopf algebras, making explicit at the
set-theoretic level a connection originating in work of Militaru \cite{Mi04} and Davydov \cite{Dav}.
In particular, \cite{CJ26x} proves that any finite set-theoretic solution gives rise to a Hopf algebra
with the so called positive basis property. Even if this is shown without using the full strength of
the classification in Theorem~\ref{main}, the argument relies on Theorem~\ref{thm:leftgroup} where we
prove that the semigroup $S$ associated to a set-theoretic solution to the PE is a left group. In
addition, \cite{CJ26x} extends the Lu--Yan--Zhu theory of Hopf algebras with positive bases
\cite{zbMATH01616308} to arbitrary fields of characteristic~$0$ (see \cite[Corollary~B]{CJ26x}).

It is then natural to ask whether the Lu--Yan--Zhu classification \cite{zbMATH01616308} could provide
an alternative proof of Theorem~\ref{main}. As explained in the remark following
\cite[Corollary~B]{CJ26x} this approach does not directly recover the set-theoretic classification. The
obstruction is twofold: first, the Hopf algebra decomposition obtained via the Fundamental Theorem of
Hopf modules does not preserve the canonical bases; second, the groups arising in the resulting
factorization appear only after a non-explicit rescaling procedure. Overcoming these obstacles would
require an explicit description of all set-theoretic bases of the relevant bicrossed product Hopf
algebra.

\section*{Acknowledgments}

Colazzo acknowledges support from the Engineering and Physical Sciences Research Council [grant number EP/V005995/1] during the early stages of this work.
The authors wish to thank Leandro Vendramin for stimulating discussion and computer calculations that started during the ``Algebra of the Yang--Baxter equation" conference in Bedlewo in 2022, from which this project originated.

\printbibliography

\end{document}